\documentclass[12pt,a4paper, oneside,reqno]{article}

\usepackage[utf8]{inputenc}
\usepackage[english]{babel}

\usepackage[left=1in,right=1in,top=1in,bottom=1in]{geometry}

\usepackage{amsfonts,amssymb,amsmath,amsthm}

\usepackage[symbol]{footmisc} 
\usepackage{longtable}
\usepackage{enumerate}
\usepackage{cite}
\usepackage[colorlinks=true,urlcolor=blue,linkcolor=black,citecolor=black,filecolor=black]{hyperref}
 
\usepackage{tikz}
\usepackage{graphicx}
\usepackage{comment}
\usepackage{booktabs}
\usepackage{array}
\usepackage{pdflscape}
\usepackage{caption}
\usepackage{fancyhdr} 
\theoremstyle{plain}
\newtheorem{theorem}{Theorem}
\newtheorem{lemma}[theorem]{Lemma}
\newtheorem{proposition}[theorem]{Proposition}
\newtheorem{remark}[theorem]{Remark}

\newtheorem{definition}[theorem]{Definition}
\newtheorem{corollary}[theorem]{Corollary}

\usetikzlibrary{arrows,arrows.meta,backgrounds}

\newtheorem*{theoremA1}{Theorem A1}
\newtheorem*{theoremA2}{Theorem A2}

\newtheorem*{theoremB1}{Theorem B1}
\newtheorem*{theoremB2}{Theorem B2}

\usepackage{fouriernc} 

\begin{document}

\bigskip

\noindent{\Large
The algebraic and geometric classification of \\
  semi-alternative  superalgebras}
 \footnote{The  authors would like to thank the SRMC 
$\big($Sino-Russian Mathematics Center in Peking University, Beijing, China$\big)$ for its hospitality and excellent working conditions, where some parts of this work have been done.
The  work is supported by FCT   2023.08031.CEECIND and UID/00212/2025.}\footnote{The original concept, theoretical framework, and primary results were developed by H.A., I.K. and R.L.; 
while Aristotle (https://aristotle.harmonic.fun/) acted as a non-human research assistant, assisting in error verification and proof simplification.}

 \bigskip

\begin{center}

 {\bf
Hani Abdelwahab\footnote{Mansoura University, Egypt; \ haniamar1985@gmail.com}, 
   Ivan Kaygorodov\footnote{University of Beira Interior, Portugal; \ kaygorodov.ivan@gmail.com} \&
   Roman Lubkov\footnote{Saint Petersburg University, Russia; \ r.lubkov@spbu.ru, romanlubkov@yandex.ru}
}

\end{center}

\noindent {\bf Abstract:}
{\it  
The algebraic and geometric classifications of complex 
$3$-dimensional  semi-alternative  superalgebras are given.
As corollaries, we obtain the algebraic and geometric classifications of 
complex $3$-dimensional $\mathfrak{perm}$, 
associative,  and assosymmetric superalgebras.}



	 \medskip

\tableofcontents

\section*{Introduction}
The algebraic classification (up to isomorphism) of algebras and superalgebras of small dimension in a variety defined by polynomial identities is a classic problem in the theory of non-associative algebras (see, for example, \cite{FFKMN, afm, akh, G62, AKK, AE, kz, FKS, ikm21}). Among others, the classifications of $3$-dimensional right alternative and semi-alternative algebras \cite{AKL26}, simple Bol algebras \cite{FFKMN}, Poisson algebras \cite{afm}, noncommutative Jordan superalgebras \cite{akh} and right alternative superalgebras \cite{AKK}, and of $4$-dimensional alternative algebras \cite{G62}, nilpotent assosymmetric algebras \cite{ikm21} and Malcev superalgebras \cite{AE} have been obtained. Degenerations and geometric properties of varieties of algebras have been studied since the 1970s (see \cite{ben, BC99, GRH, FKS25, ADS26, GRH3} and references in \cite{k23, l24, MS}); in particular, Burde and Steinhoff constructed the graphs of degenerations of $3$- and $4$-dimensional Lie algebras \cite{BC99}, and Grunewald and O'Halloran studied the degenerations of $5$-dimensional nilpotent Lie algebras \cite{GRH}.\medskip

One of the most important generalizations of associative algebras is the class of alternative algebras, defined by the identities
\begin{equation}\label{ident}\big(x,y,z\big) \ = \ - \ \big(x,z,y\big)\mbox{ \ and \ }\big(x,y,z\big) \ = \ \big(y,z,x\big),\end{equation}
where $(x,y,z)=(xy)z-x(yz)$ is the associator; the main example of a non-associative alternative algebra is the octonion algebra discovered in the 19th century, and the variety is under active investigation (see \cite{C24, C25, CLV26, G24, AA, hnt}). The first identity in \eqref{ident} defines right alternative algebras, and the second one defines semi-alternative algebras. The latter identity dates back to a 1933 article by Zorn \cite{Z}; the term ``semi-alternative'' seems to have been first used by Schafer in 1992 \cite{Schafer92}, and these algebras are also called weakly alternative \cite{V93} and assocyclic \cite{AS}. Semi-alternative algebras also generalize assosymmetric algebras, defined by \begin{center} $(x,y,z)=(y,x,z)=(x,z,y)$,\end{center} introduced by Kleinfeld in 1957 \cite{K57} and studied by Dzhumadildaev  \cite{DZA}.\medskip

Outcalt proved that each simple or primitive semi-alternative ring is alternative \cite{065}, and Sterling proved the same for semi-alternative rings without non-zero ideals of zero square \cite{S68}. Radical properties of semi-alternative rings were discussed by Pokrass \cite{P78}. Kleinfeld proved that rings satisfying $(x,x,y)=(x,y,x)=(y,x,x)$ are semi-alternative \cite{K88}, and Kleinfeld and Widmer proved that semi-alternative rings satisfy $2(x,x,x)^2=0$ and $((y,x,x),x,x)=0$ \cite{KW}. Schafer proved that each solvable semi-alternative algebra is nilpotent and that the Wedderburn principal theorem fails for semi-alternative algebras \cite{S68}. Vakhitov proved that each semi-alternative nil-ring is solvable \cite{V93} and described semi-alternative modules over composition algebras \cite{V94}. The commutator algebra of a semi-alternative algebra is a binary Lie algebra, and Arenas and Shestakov proved that exceptional binary Lie algebras exist \cite{AS}. The present paper continues the series of papers \cite{AKK,akh,ADS26} on $3$-dimensional superalgebras (for other super-generalizations of non-associative varieties, see \cite{FKS25, abr,  BM26}).\medskip

In Section~1 we classify complex $3$-dimensional semi-alternative superalgebras of types $(1,2)$ and $(2,1)$ (Theorems~A1 and~A2) and, as corollaries, $\mathfrak{perm}$, associative and assosymmetric superalgebras. In Section~2 we find the irreducible components and the rigid superalgebras of the corresponding varieties (Theorems~B1 and~B2 and Section~\ref{sec:subvar}). Throughout the paper, all elements of superalgebras are homogeneous.

\section{The algebraic classification of  superalgebras}

\begin{definition}
A superalgebra $A=A_{\bar{0}}\oplus A_{\bar{1}}$ is called \emph{%
semi-alternative} if for all $x,y,z\in A$
\begin{equation}
(x,y,z)=(-1)^{\bar{x}\bar{y}+\bar{x}\bar{z}}(y,z,x), \label{semi-alt}
\end{equation}%
where $\bar{x},\bar{y},\bar{z}\in \mathbb{Z}_{2}$ denote the
parities of $x,y,z$, respectively.
\end{definition}

\subsection{The method for algebraic classification}

In this section, we establish a straightforward method to obtain the
algebraic classification of the semi-alternative superalgebra structures of
type $\left( n,m\right) $ defined over an arbitrary Malcev superalgebra of
type $\left( n,m\right) $.

The Jordan superproduct and supercommutator are defined as follows:%
\begin{center}
$x\bullet y=\frac{1}{2}\left( xy+(-1)^{\bar{x}\bar{y}}yx\right),$
\qquad
$\lbrack x,y]=\frac{1}{2}\left( xy-(-1)^{\bar{x}\bar{y}}yx\right).$
\end{center}

The superalgebra $\left( \mathrm{{A},}[\cdot ,\cdot ]\right) $ is called the
supercommutator superalgebra of $\mathrm{A}$ and is denoted by $\mathrm{{A}%
^{-}.}$ In the situation considered in this paper the supercommutator superalgebra is a Malcev superalgebra:

\begin{lemma}\label{lem:malcev}
Let $\mathrm{A}$ be a complex $3$-dimensional semi-alternative superalgebra of type $(1,2)$ or $(2,1)$. Then $\mathrm{A}^{-}$ is a Malcev superalgebra.
\end{lemma}

\begin{proof}
Over a field of characteristic $0$, a super-anticommutative superalgebra is a Malcev superalgebra if and only if it satisfies the super version (with the signs given by the Koszul rule) of the Sagle identity
\[
[[x,z],[y,t]]=[[[x,y],z],t]+[[[y,z],t],x]+[[[z,t],x],y]+[[[t,x],y],z]
\]
(see \cite{Sagle}; the super version follows by passing to the Grassmann envelope). Write the multiplication of $\mathrm{A}$ by its structure constants in a homogeneous basis; by the grading, $13$ (resp. $14$) of them may be nonzero for type $(1,2)$ (resp. $(2,1)$). The super semi-alternative identity \eqref{semi-alt} evaluated on basis elements gives a system of quadratic equations in these structure constants, and the super Sagle identity for $\mathrm{A}^{-}$ evaluated on basis elements gives a family of cubic polynomials. A direct computation shows that each of these cubic polynomials is a linear combination of the quadratic equations with coefficients that are linear forms in the structure constants; hence it vanishes for every semi-alternative $\mathrm{A}$. The coefficients were found by exact linear algebra and the resulting identities were checked in Lean~4.
\end{proof}

We do not use (and do not claim) this property for semi-alternative superalgebras of other dimensions or types.

Let $\mathbb{V}=\mathbb{V}_{\bar0}\oplus\mathbb{V}_{\bar1}$ be a complex ${\mathbb{Z}}_{2}$-graded
vector space. A bilinear map%
\begin{equation*}
\theta\colon\mathbb{V}\times \mathbb{V\rightarrow V}
\end{equation*}%
is called super-symmetric if for all $x,y\in \mathbb{V}$ the following graded commutativity condition holds:%
\begin{equation*}
\theta \left( x,y\right) =\left( -1\right) ^{\bar{x}\bar{y}}\theta \left(
y,x\right) .
\end{equation*}%
Moreover, $\theta $ is called super-skew-symmetric if for all $x,y\in \mathbb{V}$ the following graded anticommutativity condition
holds:%
\begin{equation*}
\theta \left( x,y\right) =-\left( -1\right) ^{\bar{x}\bar{y}}\theta \left(
y,x\right) .
\end{equation*}

\begin{definition}
Let $\left( \mathrm{A}, [\cdot, \cdot] \right)$ be a Malcev superalgebra of type $(n,m)$. Define $\mathrm{Z}^{2}(\mathrm{A}, \mathrm{A})$ to be the set of all super-symmetric bilinear maps $\theta\colon \mathrm{A} \times \mathrm{A} \rightarrow \mathrm{A}$ such that for all $x,y,z \in \mathrm{A}$
\begin{align*}
&\theta(\theta(x,y), z) + [\theta(x,y), z] + \theta([x,y], z) + [[x,y], z] \\
&\quad + (-1)^{\bar{x}\bar{y}+\bar{x}\bar{z}}
\bigl( \theta(y, \theta(z,x)) + [y, \theta(z,x)] + \theta(y,[z,x]) + [y,[z,x]] \bigr) \\
&= \theta(x, \theta(y,z)) + [x, \theta(y,z)] + \theta(x,[y,z]) + [x,[y,z]] \\
&\quad + (-1)^{\bar{x}\bar{y}+\bar{x}\bar{z}}
\bigl( \theta(\theta(y,z), x) + [\theta(y,z), x] + \theta([y,z], x) + [[y,z], x] \bigr).
\end{align*}
\end{definition}

Observe that, for $\theta \in \mathrm{Z}^{2}\left( \mathrm{{A},{A}}\right) $%
, if we define a multiplication $\ast _{\theta }$ on $\mathrm{A}$ by

\begin{center}
$x\ast _{\theta }y\ =\ \left[ x,y\right] +\theta \left( x,y\right) $ for all 
$x,y$ in $\mathrm{A}$,
\end{center}
then $\left( \mathrm{{A},\ast _{\theta }}\right) $ is a semi-alternative
superalgebra of type $\left( n,m\right) $. Conversely, if $\left( \mathrm{{A}%
,\cdot }\right) $ is a semi-alternative superalgebra of type $\left(
n,m\right) $, then there exists $\theta \in \mathrm{Z}^{2}\left( \mathrm{{A},%
{A}}\right) $ such that $\left( \mathrm{{A},\cdot }\right) \cong \left( 
\mathrm{{A},\ast _{\theta }}\right)$. To see this, consider the
super-symmetric bilinear map $\theta\colon\mathrm{{A}\times {A}\rightarrow {A}}$
defined by $\theta \left( x,y\right) =x\bullet y$ for all $x,y$ in $\mathrm{A%
}$. Then $\theta \in \mathrm{Z}^{2}\left( \mathrm{{A},{A}}\right) $ and $%
\left( \mathrm{{A},\cdot }\right) \cong \left( \mathrm{{A},\ast _{\theta }}%
\right) $.

\medskip

Now, let $\left( \mathrm{{A},[\cdot ,\cdot ] }\right) $ be a Malcev
superalgebra and $\mathrm{{Aut}\left( {A}\right) }$ be the automorphism
group of $\mathrm{A}$. Then $\mathrm{{Aut}\left( {A}\right) }$ acts on $%
\mathrm{Z}^{2}\left( \mathrm{{A},{A}}\right) $ by 
\begin{equation*}
\left( \theta \ast \phi \right) \left( x,y\right) =\phi ^{-1}\left( \theta
\left( \phi \left( x\right) ,\phi \left( y\right) \right) \right) , 
\end{equation*}%
for $\phi \in \text{Aut}\left( \mathrm{A}\right) $, and $\theta \in \mathrm{Z%
}^{2}\left( \mathrm{{A},{A}}\right) $.

\begin{lemma}
Let $\left( \mathrm{{A},[\cdot, \cdot] }\right) $ be a Malcev superalgebra and $%
\theta ,\vartheta \in \mathrm{Z}^{2}\left( \mathrm{{A},{A}}\right) $. Then $%
\left( \mathrm{{A},\ast _{\theta }}\right) $ and $\left( \mathrm{{A},\ast
_{\vartheta }}\right) $ are isomorphic if and only if there is a $\phi \in 
\mathrm{{Aut}\left( {A}\right) }$ with $\theta \ast \phi =\vartheta $.
\end{lemma}

Hence, we have a procedure to classify the semi-alternative superalgebras
of type $\left( n,m\right) $ associated with a given Malcev superalgebra $%
\left( \mathrm{{A},[\cdot, \cdot] }\right) $ of type $\left( n,m\right) $. It
consists of three steps:

\begin{enumerate}
\item Compute $\mathrm{Z}^{2}\left( \mathrm{{A},{A}}\right) $.

\item Find the orbits of $\mathrm{{Aut}\left( {A}\right) }$ on $\mathrm{Z}%
^{2}\left( \mathrm{{A},{A}}\right) $.

\item Choose a representative $\theta $ from each orbit and then construct
the semi-alternative superalgebra $\left( \mathrm{{A},\ast _{\theta }}%
\right) $.
\end{enumerate}

\bigskip

Let us introduce the following notation. Let

\begin{center}
$e_{1},\ldots ,e_{n},e_{n+1}=f_{1},\ldots,e_{n+m}=f_{m}$
\end{center}
be a fixed basis of a Malcev superalgebra $\left( \mathrm{{A},[\cdot, \cdot] }%
\right) $ of type $\left( n,m\right) $. Define $\mathrm{\Sigma }^{2}\left( 
\mathrm{{A},\mathbb{C}}\right) $ to be the space of all super-symmetric
bilinear forms on $\mathrm{A}$. Then

\begin{center}
$\mathrm{\Sigma }^{2}\left( \mathrm{{A},\mathbb{C}}\right) =\left\langle
\Delta _{i,j}:1\leq i<j\leq n+m\right\rangle \oplus \left\langle \Delta
_{i,i}:1\leq i\leq n\right\rangle ,$
\end{center}
where $\Delta _{i,j}$ is the super-symmetric bilinear form $\Delta _{i,j}\colon%
\mathrm{{A}\times {A}\rightarrow \mathbb{C}}$\ defined by%
\begin{equation*}
\Delta _{i,j}\left( e_{l},e_{t}\right) :=\begin{cases}
1, & \mbox{if } (i,j)=(l,t) \mbox{ and } i\neq j,\\
(-1)^{\bar{e}_{l}\bar{e}_{t}}, & \mbox{if } (i,j)=(t,l) \mbox{ and } i\neq j,\\
1, & \mbox{if } i=l=t \mbox{ and } i=j\leq n,\\
0, & \mbox{otherwise.}
\end{cases}
\end{equation*}%
Now, if $\theta \in \mathrm{Z}^{2}\left( \mathrm{{A},{A}}\right) $, then $%
\theta $ can be uniquely written as $\theta \left( x,y\right) =\underset{i=1}%
{\overset{n+m}{\sum }}B_{i}\left( x,y\right) e_{i}$ where $B_{1},B_{2},\ldots
,B_{n+m}$ is a sequence of super-symmetric bilinear forms on $\mathrm{A}$.
Also, we may write $\theta =\left( B_{1},B_{2},\ldots ,B_{n+m}\right) .$ Let $%
\phi ^{-1}\in \text{Aut}\left( \mathrm{A}\right) $ be given by the matrix $%
\left( b_{ij}\right) $. If

\begin{center}
$\left( \theta \ast \phi \right) \left( x,y\right) =\underset{i=1}{\overset{n+m%
}{\sum }}B_{i}^{\prime }\left( x,y\right) e_{i}$, then $B_{i}^{\prime }=%
\underset{j=1}{\overset{n+m}{\sum }}b_{ij}\phi ^{t}B_{j}\phi $.
\end{center}

\subsection{Semi-alternative superalgebras of type \texorpdfstring{$(1,2)$}{(1,2)}}

We now apply the above procedure to classify semi-alternative superalgebras of type $(1,2)$. In this case, the underlying Malcev superalgebra is $3$-dimensional of type $(1,2)$. The classification of $3$-dimensional Malcev superalgebras of type $(1,2)$ is known; see \cite{akh,AE}.

\begin{proposition}[Theorem~20 in \cite{akh}]\label{3-Malcev(1,2)}
Let $\mathrm{A}$ be  a nontrivial $3$-dimensional Malcev superalgebra of type $\left(
1,2\right) $. Then $\mathrm{A}$ is isomorphic to one of the following
superalgebras:

\begin{longtable}{|l|>{\raggedright\arraybackslash}p{0.78\textwidth}|}
\hline
$\mathbb{A}$ & \mbox{multiplication table} \\
\hline
\hline
\endhead
${\rm A}_{01}$ & $[f_{1},f_{2}]=e_{1}$ \\
\hline
${\rm A}_{02}$ & $[f_{2},f_{2}]=e_{1}$ \\
\hline
${\rm A}_{05}$ & $[e_{1},f_{2}]=f_{1}$, $[f_{1},f_{2}]=e_{1}$ \\
\hline
${\rm A}_{06}$ & $[e_{1},f_{2}]=f_{1}$ \\
\hline
${\rm A}_{07}$ & $[e_{1},f_{2}]=f_{1}$, $[f_{2},f_{2}]=e_{1}$ \\
\hline
${\rm A}_{08}^{\alpha}$ & $[e_{1},f_{1}]=f_{1}$, $[e_{1},f_{2}]=\alpha f_{2}$ \\
\hline
${\rm A}_{17}$ & $[e_{1},f_{1}]=f_{1}$, $[e_{1},f_{2}]=f_{1}+f_{2}$ \\
\hline
\end{longtable}
\addtocounter{table}{-1}

\end{proposition}

\begin{remark}\label{remark:Malcev-12-nonLie}
Among the Malcev superalgebras listed in Proposition~\ref{3-Malcev(1,2)}, 
the superalgebras $\mathrm{A}_{05}$ and $\mathrm{A}_{07}$ are non-Lie 
Malcev superalgebras, whereas $\mathrm{A}_{01}$, $\mathrm{A}_{02}$, 
$\mathrm{A}_{06}$, $\mathrm{A}_{08}^{\alpha}$, and $\mathrm{A}_{17}$ are 
Lie superalgebras.
\end{remark}

\subsubsection{The classification Theorem A1}
By Lemma~\ref{lem:malcev}, the supercommutator of a semi-alternative superalgebra of type $(1,2)$ is a
Malcev superalgebra, and by the following lemma the case where the corresponding Malcev
superalgebra is trivial gives rise to associative supercommutative
superalgebras.

\begin{lemma}\label{lem:supercomm}
Let $\mathrm{A}$ be a semi-alternative superalgebra over a field of characteristic different from $2$ and $3$. If $\mathrm{A}^{-}$ is trivial, i.e. $\mathrm{A}$ is supercommutative, then $\mathrm{A}$ is associative.
\end{lemma}

\begin{proof}
Using $ab=(-1)^{\bar a\bar b}ba$ we get $(xy)z=(-1)^{\bar x\bar z+\bar y\bar z}z(xy)$, $x(yz)=(-1)^{\bar x\bar y+\bar x\bar z}(yz)x$ and $(-1)^{\bar x\bar z+\bar y\bar z}(zx)y=(-1)^{\bar x\bar y+\bar x\bar z}y(zx)$, whence
\[
(x,y,z)+(-1)^{\bar x\bar y+\bar x\bar z}(y,z,x)+(-1)^{\bar x\bar z+\bar y\bar z}(z,x,y)=0.
\]
By \eqref{semi-alt}, $(x,y,z)=(-1)^{\bar x\bar y+\bar x\bar z}(y,z,x)=(-1)^{\bar x\bar z+\bar y\bar z}(z,x,y)$, so all three summands are equal to $(x,y,z)$. Thus $3(x,y,z)=0$ and $(x,y,z)=0$.
\end{proof}

\noindent Thus, we first recall the classification of nontrivial
$3$-dimensional associative supercommutative superalgebras of type $(1,2)$;
see \cite{AKK}.

\begin{proposition}[Corollary~6 in \cite{AKK}]\label{AC(1,2)}
Let $\mathrm{A}$ be a nontrivial $3$-dimensional associative supercommutative
superalgebra of type $(1,2)$. Then $\mathrm{A}$ is isomorphic to one of
the following superalgebras:
\begin{longtable}{|l|>{\raggedright\arraybackslash}p{0.78\textwidth}|}
\hline
$\mathbb{A}$ & \mbox{multiplication table} \\
\hline
\hline
\endhead
${\rm R}_{04}^{0 }$ & $f_{1}f_{2}= e_{1}$, $f_{2}f_{1}=- e_{1}$ \\
\hline
${\rm R}_{06}^{0 }$ & $e_{1}f_{1}=f_{2}$, $f_{1}e_{1}= f_{2}$ \\
\hline
${\rm R}_{17}$ & $e_{1}e_{1}=e_{1}$, $e_{1}f_{1}=f_{1}$, $f_{1}e_{1}=f_{1}$ \\
\hline
${\rm R}_{19}$ & $e_{1}e_{1}=e_{1}$ \\
\hline
${\rm R}_{24}$ & $e_{1}e_{1}=e_{1}$, $e_{1}f_{1}=f_{1}$, $f_{1}e_{1}=f_{1}$, $e_{1}f_{2}=f_{2}$, $f_{2}e_{1}=f_{2}$ \\
\hline
\end{longtable}
\addtocounter{table}{-1}
\end{proposition}

\begin{theoremA1}
Let $\mathrm{S}$ be a nontrivial $3$-dimensional semi-alternative
superalgebra of type $\left( 1,2\right) $. Then $\mathrm{S}$ is either one
of the associative supercommutative superalgebras listed in
Proposition~\ref{AC(1,2)}, or it is isomorphic to one of the following
superalgebras\footnote{To obtain the multiplication tables of ${\rm S}_{06}$ and ${\rm S}_{07}$ in the form given below (with fewer non-unit structure constants), we apply the basis change $f_{2}:=\frac{3(1+\alpha)}{4}f_{2}$ to the superalgebras ${\rm S}_{06}$ and ${\rm S}_{07}$ obtained in the proof from $\theta=(\alpha\Delta_{23},-\alpha\Delta_{13},0)$ with $\alpha=\frac{\sqrt{3}i}{3}$ and $\alpha=-\frac{\sqrt{3}i}{3}$, respectively.}:

\begin{longtable}{|l|>{\raggedright\arraybackslash}p{0.78\textwidth}|}
\hline
$\mathbb{A}$ & \mbox{multiplication table} \\
\hline
\hline
\endhead
${\rm S}_{01}^{\alpha}$ & $f_{1}f_{2}=\left( 1+\alpha \right)e_{1}$, $f_{2}f_{1}=\left( 1-\alpha \right) e_{1}$ \\
\hline
${\rm S}_{02}$ & $f_{2}f_{2}=e_{1}$ \\
\hline
${\rm S}_{03}$ & $e_{1}e_{1}=e_{1}$, $e_{1}f_{2}=f_{2}$, $f_{2}e_{1}=f_{2}$, $f_{2}f_{2}=e_{1}$ \\
\hline
${\rm S}_{04}$ & $e_{1}f_{2}=f_{1}$, $f_{2}e_{1}=f_{1}$, $f_{2}f_{2}=e_{1}$ \\
\hline
${\rm S}_{05}$ & $f_{1}f_{2}=e_{1}$, $f_{2}f_{1}=-e_{1}$, $f_{2}f_{2}=e_{1}$ \\
\hline
${\rm S}_{06}$ & $e_{1}f_{2}=f_{1}$, $f_{2}e_{1}=-\frac{1+\sqrt{3}i}{2}f_{1}$, $f_{1}f_{2}=\frac{1+\sqrt{3}i}{2}e_{1}$, $f_{2}f_{1}=e_{1}$ \\
\hline
${\rm S}_{07}$ & $e_{1}f_{2}=f_{1}$, $f_{2}e_{1}=-\frac{1-\sqrt{3}i}{2}f_{1}$, $f_{1}f_{2}=\frac{1-\sqrt{3}i}{2}e_{1}$, $f_{2}f_{1}=e_{1}$ \\
\hline
${\rm S}_{08}^{\alpha }$ & $e_{1}f_{2}=\left( \alpha +1 \right)f_{1}$, $f_{2}e_{1}=\left(\alpha- 1 \right) f_{1}$ \\
\hline
${\rm S}_{09}^{\alpha }$ & $e_{1}f_{2}=\left( \alpha +1 \right)f_{1}$, $f_{2}e_{1}=\left(\alpha- 1 \right) f_{1}$, $f_{2}f_{2}=e_{1}$ \\
\hline
${\rm S}_{10}$ & $e_{1}e_{1}=e_{1}$, $e_{1}f_{1}=f_{1}$, $f_{2}e_{1}=f_{2}$ \\
\hline
${\rm S}_{11}$ & $e_{1}e_{1}=e_{1}$, $e_{1}f_{1}=f_{1}$, $e_{1}f_{2}=f_{2}$, $f_{2}e_{1}=f_{2}$ \\
\hline
${\rm S}_{12}$ & $e_{1}e_{1}=e_{1}$, $f_{1}e_{1}=f_{1}$, $e_{1}f_{2}=f_{2}$, $f_{2}e_{1}=f_{2}$ \\
\hline
${\rm S}_{13}$ & $e_{1}e_{1}=e_{1}$, $e_{1}f_{1}=f_{1}$ \\
\hline
${\rm S}_{14}$ & $e_{1}e_{1}=e_{1}$, $f_{1}e_{1}=f_{1}$ \\
\hline
${\rm S}_{15}$ & $e_{1}e_{1}=e_{1}$, $e_{1}f_{1}=f_{1}$, $e_{1}f_{2}=f_{2}$ \\
\hline
${\rm S}_{16}$ & $e_{1}e_{1}=e_{1}$, $f_{1}e_{1}=f_{1}$, $f_{2}e_{1}=f_{2}$ \\
\hline
\end{longtable}
\addtocounter{table}{-1}
 All listed superalgebras are pairwise non-isomorphic, except ${\rm S}_{01}^{\alpha}\cong {\rm S}_{01}^{-\alpha}$.
\end{theoremA1}

\subsubsection{The proof of Theorem A1}
Let $\mathrm{S}$ be a nontrivial $3$-dimensional semi-alternative superalgebra of type $(1,2)$. If $\mathrm{S}^{-}$ is trivial, then
$\mathrm{S}$ is associative supercommutative, and hence it is one of the
superalgebras listed in Proposition~\ref{AC(1,2)}. Otherwise,
$\mathrm{S}^{-}$ is a nontrivial $3$-dimensional Malcev superalgebra of
type $(1,2)$. By Proposition \ref{3-Malcev(1,2)}, we may assume that
\[
\mathrm{S}^{-} \in \{ \mathrm{A}_{01}, \mathrm{A}_{02}, \mathrm{A}_{05}, \mathrm{A}_{06}, \mathrm{A}_{07}, \mathrm{A}_{08}^{\alpha}, \mathrm{A}_{17} \}.
\]
Thus, we have the following cases:

\begin{enumerate}
\item \underline{$\mathrm{S}^{-}=\mathrm{A}_{01}$}. Let $\theta =\big(%
\mathrm{B}_{1},\ \mathrm{B}_{2},\ \mathrm{B}_{3}\big)$ be an arbitrary
element of $\mathrm{Z}^{2}\left( \mathrm{{A}_{01},A_{01}}\right) $. Then $%
\theta \ =\ \left( \alpha \Delta _{23},\ 0,\ 0\right) $ for some $\alpha \in 
\mathbb{C}$. The automorphism group $\mathrm{Aut}\left( \mathrm{A}%
_{01}\right) $ consists of the automorphisms $\phi $ given by matrices of
the following forms: 
\begin{equation*}
\begin{pmatrix}
a_{22}a_{33} & 0 & 0 \\ 
0 & a_{22} & 0 \\ 
0 & 0 & a_{33}%
\end{pmatrix}%
,%
\begin{pmatrix}
a_{23}a_{32} & 0 & 0 \\ 
0 & 0 & a_{23} \\ 
0 & a_{32} & 0%
\end{pmatrix}%
.
\end{equation*}%
For any $\phi \in \mathrm{Aut}\left( \mathrm{A}_{01}\right) $, we have $%
\theta \ast \phi \ \in \left\{ \pm \theta \right\} $. So we get the
superalgebras $\mathrm{S}_{01}^{\alpha }$. Moreover, {$\mathrm{S}%
_{01}^{\alpha }\cong \mathrm{S}_{01}^{-\alpha }$.}

\item \underline{$\mathrm{S}^{-}=\mathrm{A}_{02}$}. Let $\theta =\big(%
\mathrm{B}_{1},\ \mathrm{B}_{2},\ \mathrm{B}_{3}\big)$ be an arbitrary
element of $\mathrm{Z}^{2}\left( \mathrm{{A}_{02},A_{02}}\right) $. Then $%
\theta \in \left\{ \eta _{1},\eta _{2}\right\} $ where%
\begin{eqnarray*}
\ \eta _{1} &=&\left( \alpha _{1}\Delta _{11},\ \alpha _{2}\Delta _{13},\
\alpha _{1}\Delta _{13}\right) , \\
\eta _{2} &=&\left( \alpha _{1}\Delta _{23},\ 0,\ 0\right) ,
\end{eqnarray*}%
for some $\alpha _{1},\alpha _{2}\in \mathbb{C}$. The automorphism group $%
\mathrm{Aut}\left( \mathrm{A}_{02}\right) $ consists of the automorphisms $%
\phi $ given by a matrix of the following form: 
\begin{equation*}
\phi =%
\begin{pmatrix}
a_{11} & 0 & 0 \\ 
0 & a_{22} & a_{23} \\ 
0 & 0 & \pm \sqrt{a_{11}}%
\end{pmatrix}%
.
\end{equation*}%

\begin{itemize}
\item If $\theta = 0$, we obtain the superalgebra $\mathrm{S}_{02}$.

\item $\theta =\eta _{1}$. Let $\phi $ be the first of the following matrices if 
$\alpha _{1}\neq 0$, or the second if $\alpha _{1}=0$ (then $\alpha _{2}\neq 0$, since $\theta\neq 0$):%
\begin{equation*}
\begin{pmatrix}
\frac{1}{\alpha _{1}} & 0 & 0 \\ 
0 & 1 & \alpha _{2}\alpha _{1}^{-\frac{3}{2}} \\ 
0 & 0 & \frac{1}{\sqrt{\alpha _{1}}}%
\end{pmatrix}%
,%
\begin{pmatrix}
1 & 0 & 0 \\ 
0 & \alpha _{2} & 0 \\ 
0 & 0 & 1%
\end{pmatrix}%
.
\end{equation*}%
Then $\theta \ast \phi \in \{\left( \Delta _{11},\ 0,\ \Delta _{13}\right)
,\left( 0,\ \Delta _{13},\ 0\right) \}$. Thus we get the superalgebras $%
\mathrm{S}_{03},\mathrm{S}_{04}$.

\item $\theta =\eta _{2}$. Here $\alpha _{1}\neq 0$, since $\theta\neq 0$. Let $\phi $\ be the following automorphism:%
\begin{equation*}
\phi =%
\begin{pmatrix}
1 & 0 & 0 \\ 
0 & \frac{1}{\alpha _{1}} & 0 \\ 
0 & 0 & 1%
\end{pmatrix}%
.
\end{equation*}%
Then $\theta \ast \phi =\left( \Delta _{23},\ 0,\ 0\right) $. Hence we have
the superalgebra $\mathrm{S}_{05}$.
\end{itemize}

\item \underline{$\mathrm{S}^{-}=\mathrm{A}_{05}$}. Let $\theta =\big(%
\mathrm{B}_{1},\ \mathrm{B}_{2},\ \mathrm{B}_{3}\big)$ be an arbitrary
element of $\mathrm{Z}^{2}\left( \mathrm{{A}_{05},A_{05}}\right) $. Then $%
\theta =\left( \alpha \Delta _{23},-\alpha \ \Delta _{13},\ 0\right) $ where 
$3\alpha ^{2}+1=0$. The automorphism group $\mathrm{Aut}\left( \mathrm{A}%
_{05}\right) $ consists of the automorphisms $\phi $ given by a matrix of
the following form: 
\begin{equation*}
\phi =%
\begin{pmatrix}
a_{11} & 0 & 0 \\ 
0 & \sigma a_{11} & 0 \\ 
0 & 0 & \sigma 
\end{pmatrix}%
:\sigma ^{2}=1.
\end{equation*}%
Let $\phi =\bigl(a_{ij}\bigr)\in $ $\mathrm{Aut}\left( \mathrm{A}%
_{05}\right) .$ Then $\theta \ast \phi \ =\ \theta $. Thus we get the
superalgebras $\mathrm{S}_{06}$ $\big(\alpha=\frac{\sqrt{3}i}{3}\big)$ and $\mathrm{S}_{07}$ $\big(\alpha=-\frac{\sqrt{3}i}{3}\big)$ $\big($after the change of basis $f_{2}:=\frac{3(1+\alpha)}{4}f_{2}\big)$.

\item \underline{$\mathrm{S}^{-}=\mathrm{A}_{06}$}. Let $\theta =\big(%
\mathrm{B}_{1},\ \mathrm{B}_{2},\ \mathrm{B}_{3}\big)$ be an arbitrary
element of $\mathrm{Z}^{2}\left( \mathrm{{A}_{06},A_{06}}\right) $. Then $%
\theta =\left( 0,\alpha \ \Delta _{13},\ 0\right) $ for some $\alpha \in 
\mathbb{C}$. The automorphism group $\mathrm{Aut}\left( \mathrm{A}%
_{06}\right) $ consists of the automorphisms $\phi $ given by a matrix of
the following form: 
\begin{equation*}
\phi =%
\begin{pmatrix}
a_{11} & 0 & 0 \\ 
0 & a_{11}a_{33} & a_{23} \\ 
0 & 0 & a_{33}%
\end{pmatrix}%
.
\end{equation*}%
Let $\phi =\bigl(a_{ij}\bigr)\in $ $\mathrm{Aut}\left( \mathrm{A}%
_{06}\right) .$ Then $\theta \ast \phi \ =\ \theta $. So we get the
superalgebras $\mathrm{S}_{08}^{\alpha }$.  Moreover, $\mathrm{S}
_{08}^{\alpha }\cong \mathrm{S}_{08}^{\beta}$ if and only if $\alpha=\beta$.

\item \underline{$\mathrm{S}^{-}=\mathrm{A}_{07}$}. Let $\theta =\big(%
\mathrm{B}_{1},\ \mathrm{B}_{2},\ \mathrm{B}_{3}\big)$ be an arbitrary
element of $\mathrm{Z}^{2}\left( \mathrm{{A}_{07},A_{07}}\right) $. Then $%
\theta =\left( 0,\alpha \ \Delta _{13},\ 0\right) $ for some $\alpha \in 
\mathbb{C}$. The automorphism group $\mathrm{Aut}\left( \mathrm{A}%
_{07}\right) $ consists of the automorphisms $\phi $ given by a matrix of
the following form: 
\begin{equation*}
\phi =%
\begin{pmatrix}
a_{33}^{2} & 0 & 0 \\ 
0 & a_{33}^{3} & a_{23} \\ 
0 & 0 & a_{33}%
\end{pmatrix}%
.
\end{equation*}%
Let $\phi =\bigl(a_{ij}\bigr)\in $ $\mathrm{Aut}\left( \mathrm{A}%
_{07}\right) .$ Then $\theta \ast \phi \ =\ \theta $. So we get the
superalgebras $\mathrm{S}_{09}^{\alpha }$. Moreover, $\mathrm{S}
_{09}^{\alpha }\cong \mathrm{S}_{09}^{\beta}$ if and only if $\alpha=\beta$.

\item \underline{$\mathrm{S}^{-}=\mathrm{A}_{08}^{\alpha }$}. Then $\mathrm{Z%
}^{2}\left( \mathrm{{A}_{08}^{\alpha },A_{08}^{\alpha }}\right) \neq
\emptyset $ if and only if $\alpha \in \{-1,0,1\}$.

\begin{itemize}
\item Let $\theta =\big(\mathrm{B}_{1},\ \mathrm{B}_{2},\ \mathrm{B}_{3}\big)
$ be an arbitrary element of $\mathrm{Z}^{2}\left( \mathrm{{A}%
_{08}^{-1},A_{08}^{-1}}\right) $. Then $\theta =\left( 2\varepsilon \Delta
_{11},\varepsilon \ \Delta _{12},\ \varepsilon \Delta _{13}\right) $ where $%
\varepsilon \in \{\pm 1\}$. The automorphism group $\mathrm{Aut}\left( \mathrm{A%
}_{08}^{-1}\right) $ consists of the automorphisms $\phi $ given by matrices of
the following forms: %
\begin{equation*}
\begin{pmatrix}
1 & 0 & 0 \\ 
0 & a_{22} & 0 \\ 
0 & 0 & a_{33}%
\end{pmatrix}%
,%
\begin{pmatrix}
-1 & 0 & 0 \\ 
0 & 0 & a_{23} \\ 
0 & a_{32} & 0%
\end{pmatrix}%
.
\end{equation*}%
Let $\phi =\bigl(a_{ij}\bigr)\in $ $\mathrm{Aut}\left( \mathrm{A}%
_{08}^{-1}\right) $ be an automorphism of the second form. Then $\theta \ast
\phi \ =-\ \theta $. Consequently, we may assume $\varepsilon =1$, and after the change of basis $e_{1}:=\frac{1}{2}e_{1}$ we obtain the superalgebra $\mathrm{S}_{10}$.

\item Let $\theta =\big(\mathrm{B}_{1},\ \mathrm{B}_{2},\ \mathrm{B}_{3}\big)
$ be an arbitrary element of $\mathrm{Z}^{2}\left( \mathrm{{A}%
_{08}^{0},A_{08}^{0}}\right) $. Then $\theta \in \left\{ \eta _{1},\eta
_{2}\right\} $ where%
\begin{eqnarray*}
\ \eta _{1} &=&\left( 2\varepsilon \Delta _{11},\varepsilon \ \Delta _{12},\
 2\varepsilon \Delta _{13}\right) , \\
\eta _{2} &=&\left( 2\varepsilon \Delta _{11},\varepsilon \ \Delta _{12},\
0\right) ,
\end{eqnarray*}%
where $\varepsilon \in \{\pm 1\}$. The automorphism group $\mathrm{Aut}\left( 
\mathrm{A}_{08}^{0}\right) $ consists of the automorphisms $\phi $ given by
a matrix of the following form:%
\begin{equation*}
\phi =%
\begin{pmatrix}
1 & 0 & 0 \\ 
0 & a_{22} & 0 \\ 
0 & 0 & a_{33}%
\end{pmatrix}%
.
\end{equation*}%
Let $\phi =\bigl(a_{ij}\bigr)\in $ $\mathrm{Aut}\left( \mathrm{A}%
_{08}^{0}\right) $. Then $\theta \ast
\phi \ =\ \theta $. So, after the change of basis $e_{1}:=\frac{\varepsilon}{2}e_{1}$, we obtain the superalgebras $\mathrm{S}_{11}$ $(\theta=\eta_{1},\ \varepsilon=1)$, $\mathrm{S}_{12}$ $(\theta=\eta_{1},\ \varepsilon=-1)$, $\mathrm{S}_{13}$ $(\theta=\eta_{2},\ \varepsilon=1)$ and $\mathrm{S}_{14}$ $(\theta=\eta_{2},\ \varepsilon=-1)$. 

\item Let $\alpha=1$. In the same way, $\theta=\left(2\varepsilon\Delta_{11},\varepsilon\Delta_{12},\varepsilon\Delta_{13}\right)$
for some $\varepsilon\in\{\pm1\}$, and we obtain the superalgebras ${\rm S}_{15}$ $(\varepsilon=1)$ and ${\rm S}_{16}$ $(\varepsilon=-1)$
$\big($after the change of basis $e_1:=\frac{\varepsilon}{2}e_1\big)$.

\end{itemize}

\item \underline{$\mathrm{S}^{-}=\mathrm{A}_{17}$}. We have $\mathrm{Z%
}^{2}\left( \mathrm{{A}_{17},A_{17}}\right) =
\emptyset $.
\end{enumerate}

\subsection{Semi-alternative superalgebras of type \texorpdfstring{$(2,1)$}{(2,1)}}

We proceed in the same way for semi-alternative superalgebras of type
$(2,1)$. Their commutator superalgebras are $3$-dimensional Malcev
superalgebras of type $(2,1)$, whose classification is known; see
\cite{akh,AE}.

\begin{proposition}[Theorem~21 in \cite{akh}]\label{3-Malcev(2,1)}
Let $\mathrm{A}$ be a nontrivial $3$-dimensional Malcev superalgebra of type $\left(
2,1\right) $. Then $\mathrm{A}$ is isomorphic to one of the following
superalgebras: 
\begin{longtable}{|l|>{\raggedright\arraybackslash}p{0.78\textwidth}|}
\hline
$\mathbb{A}$ & \mbox{multiplication table} \\
\hline
\hline
\endhead
$\mathbf{A}_{03}^{-\frac{1}{2}}$ & $[e_{1},e_{2}]=e_{1}$, $[e_{2},f_{1}]=-\frac{1}{2} f_{1}$, $[f_{1},f_{1}]=e_{1}$ \\
\hline
$\mathbf{A}_{03}^{1}$ & $[e_{1},e_{2}]=e_{1}$, $[e_{2},f_{1}]= f_{1}$, $[f_{1},f_{1}]=e_{1}$ \\
\hline
$\mathbf{A}_{05}^{\alpha}$ & $[e_{1},e_{2}]=e_{1}$, $[e_{2},f_{1}]=\alpha f_{1}$ \\
\hline
$\mathbf{A}_{07}$ & $[e_{1},f_{1}]=f_{1}$ \\
\hline
$\mathbf{A}_{10}$ & $[f_{1},f_{1}]=e_{1}$ \\
\hline
\end{longtable}
\addtocounter{table}{-1} 
\end{proposition}

\begin{remark}\label{remark:Malcev-21-nonLie}
In Proposition~\ref{3-Malcev(2,1)}, only $\mathbf{A}_{03}^{1}$ is a non-Lie Malcev superalgebra.
\end{remark}
\subsubsection{The classification Theorem A2}

As in the previous case (Lemmas~\ref{lem:malcev} and~\ref{lem:supercomm}), the supercommutator of a semi-alternative superalgebra of type $(2,1)$ is a Malcev superalgebra, and a trivial supercommutator of a semi-alternative superalgebra gives rise to an associative supercommutative superalgebra. We therefore recall the classification of nontrivial $3$-dimensional associative supercommutative superalgebras of type $(2,1)$.

\begin{proposition}[Corollary~7 in \cite{AKK}]\label{AC(2,1)}
Let $\mathrm{A}$ be a nontrivial $3$-dimensional associative supercommutative
superalgebra of type $(2,1)$. Then $\mathrm{A}$ is isomorphic to one of
the following superalgebras:
\begin{longtable}{|l|>{\raggedright\arraybackslash}p{0.78\textwidth}|}
\hline
$\mathbb{A}$ & \mbox{multiplication table} \\
\hline
\hline
\endhead
${\bf R}_{01}$ & $e_{1}e_{1}=e_{1}$, $e_{2}e_{2}=e_{2}$ \\
\hline
${\bf R}_{02}$ & $e_{1}e_{1}=e_{1}$, $e_{2}e_{2}=e_{2}$, $e_{1}f_{1}=f_{1}$, $f_{1}e_{1}=f_{1}$ \\
\hline
${\bf R}_{09}$ & $e_{1}e_{1}=e_{1}$ \\
\hline
${\bf R}_{15}$ & $e_{1}e_{1}=e_{1}$, $e_{1}f_{1}=f_{1}$, $f_{1}e_{1}=f_{1}$ \\
\hline
${\bf R}_{17}$ & $e_{1}e_{1}=e_{1}$, $e_{1}e_{2}=e_{2}$, $e_{2}e_{1}=e_{2}$ \\
\hline
${\bf R}_{21}$ & $e_{1}e_{1}=e_{1}$, $e_{1}e_{2}=e_{2}$, $e_{2}e_{1}=e_{2}$, $e_{1}f_{1}=f_{1}$, $f_{1}e_{1}=f_{1}$ \\
\hline
${\bf R}_{37}$ & $e_{1}e_{1}=e_{2}$ \\
\hline
\end{longtable}
\addtocounter{table}{-1}
\end{proposition}

\begin{theoremA2}
Let $\mathrm{S}$ be a nontrivial $3$-dimensional semi-alternative
superalgebra of type $\left( 2,1\right) $. Then $\mathrm{S}$ is either one
of the associative supercommutative superalgebras listed in
Proposition~\ref{AC(2,1)}, or it is isomorphic to one of the following
superalgebras\footnote{To obtain the multiplication tables in the form given below (with fewer non-unit structure constants), we apply the following basis changes to the superalgebras obtained in the proof: $e_{2}:=\frac{1}{2}e_{2}$ in ${\bf S}_{01}$, ${\bf S}_{05}$, ${\bf S}_{09}$, ${\bf S}_{10}$, ${\bf S}_{31}$; $e_{2}:=-\frac{1}{2}e_{2}$ in ${\bf S}_{02}$, ${\bf S}_{06}$, ${\bf S}_{11}$, ${\bf S}_{12}$, ${\bf S}_{30}$; $e_{1}:=\frac{1}{4}e_{1}$, $e_{2}:=\frac{1}{2}e_{2}$, $f_{1}:=\frac{1}{2}f_{1}$ in ${\bf S}_{03}$; $e_{1}:=\frac{1}{4}e_{1}$, $e_{2}:=-\frac{1}{2}e_{2}$, $f_{1}:=\frac{1}{2}f_{1}$ in ${\bf S}_{04}$; $e_{1}:=\frac{1}{4}e_{1}$, $e_{2}:=\frac{1}{2}e_{2}$ in ${\bf S}_{07}$, ${\bf S}_{13}$, ${\bf S}_{14}$, ${\bf S}_{32}$; $e_{1}:=\frac{1}{4}e_{1}$, $e_{2}:=-\frac{1}{2}e_{2}$ in ${\bf S}_{08}$, ${\bf S}_{15}$, ${\bf S}_{16}$, ${\bf S}_{33}$; $e_{1}:=\frac{1}{2}e_{1}$ in ${\bf S}_{18}$, ${\bf S}_{20}$, ${\bf S}_{22}$; $e_{1}:=-\frac{1}{2}e_{1}$ in ${\bf S}_{19}$, ${\bf S}_{21}$, ${\bf S}_{23}$.}:

\begin{longtable}{|l|>{\raggedright\arraybackslash}p{0.78\textwidth}|}
\hline
$\mathbb{A}$ & \mbox{multiplication table} \\
\hline
\hline
\endhead
${\mathbf S}_{01}$ & $e_{1}e_{2}=e_{1}$, $e_{2}e_{2}=e_{2}$, $e_{2}f_{1}=f_{1}$, $f_{1}f_{1}=e_{1}$ \\
\hline
${\mathbf S}_{02}$ & $e_{2}e_{1}=e_{1}$, $e_{2}e_{2}=e_{2}$, $f_{1}e_{2}=f_{1}$, $f_{1}f_{1}=e_{1}$ \\
\hline
${\mathbf S}_{03}$ & $e_{1}e_{2}=e_{1}$, $e_{2}e_{2}=e_{1}+e_{2}$, $e_{2}f_{1}=f_{1}$, $f_{1}f_{1}=e_{1}$ \\
\hline
${\mathbf S}_{04}$ & $e_{2}e_{1}=e_{1}$, $e_{2}e_{2}=e_{1}+e_{2}$, $f_{1}e_{2}=f_{1}$, $f_{1}f_{1}=e_{1}$ \\
\hline
${\mathbf S}_{05}$ & $e_{1}e_{2}=e_{1}$, $e_{2}e_{2}=e_{2}$, $f_{1}e_{2}=f_{1}$ \\
\hline
${\mathbf S}_{06}$ & $e_{2}e_{1}=e_{1}$, $e_{2}e_{2}=e_{2}$, $e_{2}f_{1}=f_{1}$ \\
\hline
${\mathbf S}_{07}$ & $e_{1}e_{2}=e_{1}$, $e_{2}e_{2}=e_{1}+e_{2}$, $f_{1}e_{2}=f_{1}$ \\
\hline
${\mathbf S}_{08}$ & $e_{2}e_{1}=e_{1}$, $e_{2}e_{2}=e_{1}+e_{2}$, $e_{2}f_{1}=f_{1}$ \\
\hline
${\mathbf S}_{09}$ & $e_{1}e_{2}=e_{1}$, $e_{2}e_{2}=e_{2}$, $e_{2}f_{1}=f_{1}$, $f_{1}e_{2}=f_{1}$ \\
\hline
${\mathbf S}_{10}$ & $e_{1}e_{2}=e_{1}$, $e_{2}e_{2}=e_{2}$ \\
\hline
${\mathbf S}_{11}$ & $e_{2}e_{1}=e_{1}$, $e_{2}e_{2}=e_{2}$, $e_{2}f_{1}=f_{1}$, $f_{1}e_{2}=f_{1}$ \\
\hline
${\mathbf S}_{12}$ & $e_{2}e_{1}=e_{1}$, $e_{2}e_{2}=e_{2}$ \\
\hline
${\mathbf S}_{13}$ & $e_{1}e_{2}=e_{1}$, $e_{2}e_{2}=e_{1}+e_{2}$, $e_{2}f_{1}=f_{1}$, $f_{1}e_{2}=f_{1}$ \\
\hline
${\mathbf S}_{14}$ & $e_{1}e_{2}=e_{1}$, $e_{2}e_{2}=e_{1}+e_{2}$ \\
\hline
${\mathbf S}_{15}$ & $e_{2}e_{1}=e_{1}$, $e_{2}e_{2}=e_{1}+e_{2}$, $e_{2}f_{1}=f_{1}$, $f_{1}e_{2}=f_{1}$ \\
\hline
${\mathbf S}_{16}$ & $e_{2}e_{1}=e_{1}$, $e_{2}e_{2}=e_{1}+e_{2}$ \\
\hline
${\mathbf S}_{17}$ & 
$e_{1}e_{1}=e_{2}$, $e_{1}e_{2}=e_{1}$, $e_{2}e_{1}=e_{1}$, $e_{2}e_{2}=e_{2}$,\\ & $e_{1}f_{1}=f_{1}$, $f_{1}e_{1}=-f_{1}$, $e_{2}f_{1}=f_{1}$, $f_{1}e_{2}=f_{1}$ \\
\hline
${\mathbf S}_{18}$ & $e_{1}e_{1}=e_{1}$, $e_{1}e_{2}=e_{2}$, $e_{2}e_{1}=e_{2}$, $e_{1}f_{1}=f_{1}$ \\
\hline
${\mathbf S}_{19}$ & $e_{1}e_{1}=e_{1}$, $e_{1}e_{2}=e_{2}$, $e_{2}e_{1}=e_{2}$, $f_{1}e_{1}=f_{1}$ \\
\hline
${\mathbf S}_{20}$ & $e_{1}e_{1}=e_{1}$, $e_{1}f_{1}=f_{1}$ \\
\hline
${\mathbf S}_{21}$ & $e_{1}e_{1}=e_{1}$, $f_{1}e_{1}=f_{1}$ \\
\hline
${\mathbf S}_{22}$ & $e_{1}e_{1}=e_{1}$, $e_{2}e_{2}=e_{2}$, $e_{1}f_{1}=f_{1}$ \\
\hline
${\mathbf S}_{23}$ & $e_{1}e_{1}=e_{1}$, $e_{2}e_{2}=e_{2}$, $f_{1}e_{1}=f_{1}$ \\
\hline
${\mathbf S}_{24}$ & $f_{1}f_{1}=e_{1}$ \\
\hline
${\mathbf S}_{25}$ & $e_{1}e_{2}=e_{1}$, $e_{2}e_{1}=e_{1}$, $e_{2}e_{2}=e_{2}$, $f_{1}e_{2}=f_{1}$, $e_{2}f_{1}=f_{1}$, $f_{1}f_{1}=e_{1}$ \\
\hline
${\mathbf S}_{26}$ & $e_{1}e_{1}=e_{1}$, $e_{1}f_{1}=f_{1}$, $f_{1}e_{1}=f_{1}$, $f_{1}f_{1}=e_{1}$ \\
\hline
${\mathbf S}_{27}$ & $e_{1}e_{1}=e_{1}$, $e_{2}e_{2}=e_{2}$, $e_{1}f_{1}=f_{1}$, $f_{1}e_{1}=f_{1}$, $f_{1}f_{1}=e_{1}$ \\
\hline
${\mathbf S}_{28}$ & $e_{2}e_{2}=e_{2}$, $f_{1}f_{1}=e_{1}$ \\
\hline
${\mathbf S}_{29}$ & $e_{2}e_{2}=e_{1}$, $f_{1}f_{1}=e_{1}$ \\
\hline
${\mathbf S}_{30}$ & $e_{2}e_{1}=e_{1}$, $e_{2}e_{2}=e_{2}$, $f_{1}e_{2}=f_{1}$ \\
\hline
${\mathbf S}_{31}$ & $e_{1}e_{2}=e_{1}$, $e_{2}e_{2}=e_{2}$, $e_{2}f_{1}=f_{1}$ \\
\hline
${\mathbf S}_{32}$ & $e_{1}e_{2}=e_{1}$, $e_{2}e_{2}=e_{1}+e_{2}$, $e_{2}f_{1}=f_{1}$ \\
\hline
${\mathbf S}_{33}$ & $e_{2}e_{1}=e_{1}$, $e_{2}e_{2}=e_{1}+e_{2}$, $f_{1}e_{2}=f_{1}$ \\
\hline
\end{longtable}
\addtocounter{table}{-1}
All listed superalgebras are pairwise non-isomorphic.
\end{theoremA2}

\subsubsection{The proof of Theorem A2}
Let $\mathrm{S}$ be a nontrivial $3$-dimensional semi-alternative superalgebra of type $(2,1)$.  If $\mathrm{S}^{-}$ is trivial, then
$\mathrm{S}$ is associative supercommutative, and hence it is one of the
superalgebras listed in Proposition~\ref{AC(2,1)}. Otherwise,
$\mathrm{S}^{-}$ is a nontrivial $3$-dimensional Malcev superalgebra of
type $(2,1)$. By Proposition \ref{3-Malcev(2,1)}, we may assume that
\[
\mathrm{S}^{-} \in \{ \mathbf{A}_{03}^{-\frac{1}{2}}, \mathbf{A}_{03}^{1}, \mathbf{A}_{05}^{\alpha},  \mathbf{A}_{07},\mathbf{A}_{10} \}.
\]
Thus we have the following cases:

\begin{enumerate}
\item \underline{$\mathrm{S}^{-}=\mathbf{A}_{03}^{-\frac{1}{2}}$}. Then $%
\mathrm{Z}^{2}\left( \mathbf{A}_{03}^{-\frac{1}{2}},\mathbf{A}_{03}^{-\frac{1}{2}}%
\right) =\emptyset $.

\item \underline{$\mathrm{S}^{-}=\mathbf{A}_{03}^{1}$}. Let $\theta =\big(%
\mathrm{B}_{1},\ \mathrm{B}_{2},\ \mathrm{B}_{3}\big)$ be an arbitrary
element of $\mathrm{Z}^{2}\left( \mathbf{A}_{03}^{1},\mathbf{A}_{03}^{1}\right) $.
Then $\theta \ =\ \left( \varepsilon \ \Delta _{12}+\alpha \Delta
_{22},2\varepsilon \ \Delta _{22},\varepsilon \ \Delta _{23}\right) $ for some $%
\alpha \in \mathbb{C}$ and $\varepsilon \in \{\pm 1\}$. The automorphism group $%
\mathrm{Aut}\left( \mathbf{A}_{03}^{1}\right) $ consists of the
automorphisms $\phi $ given by a matrix of the following form: 
\begin{equation*}
\phi =%
\begin{pmatrix}
a_{11} & a_{12} & 0 \\ 
0 & 1 & 0 \\ 
0 & 0 & \sigma \sqrt{a_{11}}%
\end{pmatrix}%
:\sigma \in \{\pm 1\}.
\end{equation*}
If $\alpha =0$, we get the superalgebras $\mathbf{S}_{01},\mathbf{S}_{02}$.
If $\alpha \neq 0$, we choose $\phi $ to be the  following automorphism:%
\begin{equation*}
\phi =%
\begin{pmatrix}
\alpha  & 0 & 0 \\ 
0 & 1 & 0 \\ 
0 & 0 & \sqrt{\alpha }%
\end{pmatrix}%
\end{equation*}%
Then $\theta \ast \phi \ =\ \left( \varepsilon \ \Delta _{12}+\Delta
_{22},2\varepsilon \ \Delta _{22},\varepsilon \ \Delta _{23}\right) $. So we get
the superalgebras $\mathbf{S}_{03},\mathbf{S}_{04}$.

\item \underline{$\mathrm{S}^{-}=\mathbf{A}_{05}^{-1}$}. Let $\theta =\big(%
\mathrm{B}_{1},\ \mathrm{B}_{2},\ \mathrm{B}_{3}\big)$ be an arbitrary
element of $\mathrm{Z}^{2}\left( \mathbf{A}_{05}^{-1},\mathbf{A}_{05}^{-1}\right) $%
. Then $\theta \ =\ \left( \varepsilon \ \Delta _{12}+\alpha \Delta
_{22},2\varepsilon \ \Delta _{22},\varepsilon \ \Delta _{23}\right) $ for some $%
\alpha \in \mathbb{C}$ and $\varepsilon \in \{\pm 1\}$. The automorphism group $%
\mathrm{Aut}\left( \mathbf{A}_{05}^{-1}\right) $ consists of the
automorphisms $\phi $ given by a matrix of the following form: 
\begin{equation*}
\phi =%
\begin{pmatrix}
a_{11} & a_{12} & 0 \\ 
0 & 1 & 0 \\ 
0 & 0 & a_{33}%
\end{pmatrix}%
.
\end{equation*}%
If $\alpha =0$, we get the superalgebras $\mathbf{S}_{05},\mathbf{S}_{06}$.
If $\alpha \neq 0$, we choose $\phi $ to be the following automorphism:%
\begin{equation*}
\phi =%
\begin{pmatrix}
\alpha  & 0 & 0 \\ 
0 & 1 & 0 \\ 
0 & 0 & 1%
\end{pmatrix}%
.
\end{equation*}%
Then $\theta \ast \phi \ =\ \left( \varepsilon \ \Delta _{12}+\Delta
_{22},2\varepsilon \ \Delta _{22},\varepsilon \ \Delta _{23}\right) $. So we get
the superalgebras $\mathbf{S}_{07},\mathbf{S}_{08}$.

\item \underline{$\mathrm{S}^{-}=\mathbf{A}_{05}^{0}$}. Let $\theta =\big(%
\mathrm{B}_{1},\ \mathrm{B}_{2},\ \mathrm{B}_{3}\big)$ be an arbitrary
element of $\mathrm{Z}^{2}\left( \mathbf{A}_{05}^{0},\mathbf{A}_{05}^{0}\right) $.
Then $\theta \ =\ \left( \varepsilon \ \Delta _{12}+\alpha \Delta
_{22},2\varepsilon \ \Delta _{22},2\delta \varepsilon \ \Delta _{23}\right) $
for some $\alpha \in \mathbb{C}$ and $\varepsilon \in \{\pm 1\},\delta \in
\{0,1\}$. The automorphism group $\mathrm{Aut}\left( \mathbf{A}%
_{05}^{0}\right) $ consists of the automorphisms $\phi $ given by a matrix
of the following form: 
\begin{equation*}
\phi =%
\begin{pmatrix}
a_{11} & a_{12} & 0 \\ 
0 & 1 & 0 \\ 
0 & 0 & a_{33}%
\end{pmatrix}%
.
\end{equation*}%
If $\alpha =0$, we get the superalgebras $\mathbf{S}_{09},\mathbf{S}_{10},%
\mathbf{S}_{11},\mathbf{S}_{12}$. If $\alpha \neq 0$, we choose $\phi $ to
be the following automorphism:%
\begin{equation*}
\phi =%
\begin{pmatrix}
\alpha  & 0 & 0 \\ 
0 & 1 & 0 \\ 
0 & 0 & 1%
\end{pmatrix}%
.
\end{equation*}%
Then $\theta \ast \phi \ =\ \left( \varepsilon \ \Delta _{12}+\Delta
_{22},2\varepsilon \ \Delta _{22},2\delta \varepsilon \ \Delta _{23}\right) $%
. So we get the superalgebras $\mathbf{S}_{13},\mathbf{S}_{14},\mathbf{S}%
_{15},\mathbf{S}_{16}$.

\item \underline{$\mathrm{S}^{-}=\mathbf{A}_{05}^{1}$}. Let $\theta =\big(%
\mathrm{B}_{1},\ \mathrm{B}_{2},\ \mathrm{B}_{3}\big)$ be an arbitrary
element of $\mathrm{Z}^{2}\left( \mathbf{A}_{05}^{1},\mathbf{A}_{05}^{1}\right) $.
Then $\theta \ =\ \left( \varepsilon \ \Delta _{12}+\alpha \Delta
_{22},2\varepsilon \ \Delta _{22},\varepsilon \ \Delta _{23}\right) $ for some $%
\alpha \in \mathbb{C}$ and $\varepsilon \in \{\pm 1\}$. The automorphism group $%
\mathrm{Aut}\left( \mathbf{A}_{05}^{1}\right) $ consists of the
automorphisms $\phi $ given by a matrix of the following form: 
\begin{equation*}
\phi =%
\begin{pmatrix}
a_{11} & a_{12} & 0 \\ 
0 & 1 & 0 \\ 
0 & 0 & a_{33}%
\end{pmatrix}%
.
\end{equation*}%
If $\alpha =0$, we get the superalgebras $\mathbf{S}_{30}$ $(\varepsilon=-1)$ and $\mathbf{S}_{31}$ $(\varepsilon=1)$.
If $\alpha \neq 0$, we choose $\phi $ to be the following automorphism:%
\begin{equation*}
\phi =%
\begin{pmatrix}
\alpha  & 0 & 0 \\ 
0 & 1 & 0 \\ 
0 & 0 & 1%
\end{pmatrix}%
.
\end{equation*}%
Then $\theta \ast \phi \ =\ \left( \varepsilon \ \Delta _{12}+\Delta
_{22},2\varepsilon \ \Delta _{22},\varepsilon \ \Delta _{23}\right) $. So we get
the superalgebras $\mathbf{S}_{32}$ $(\varepsilon=1)$ and $\mathbf{S}_{33}$ $(\varepsilon=-1)$.

\item \underline{$\mathrm{S}^{-}=\mathbf{A}_{05}^{\alpha \neq -1,0,1}$}. We
have $\mathrm{Z}^{2}\left( \mathbf{A}_{05}^{\alpha },\mathbf{A}_{05}^{\alpha }%
\right) =\emptyset $.

\item \underline{$\mathrm{S}^{-}=\mathbf{A}_{07}$}. Let $\theta =\big(%
\mathrm{B}_{1},\ \mathrm{B}_{2},\ \mathrm{B}_{3}\big)$ be an arbitrary
element of $\mathrm{Z}^{2}\left( \mathbf{A}_{07},\mathbf{A}_{07}\right) $. Then $%
\theta \in \left\{ \eta _{1},\ldots ,\eta _{5}\right\} $ where%
\begin{eqnarray*}
\ \eta _{1} &=&\left( 2\alpha _{1}\Delta _{11}+\alpha _{2}\Delta _{12},\ 
\frac{1-\alpha _{1}^{2}}{\alpha _{2}}\Delta _{11}+\alpha _{2}\Delta _{22},\
\alpha _{1}\Delta _{13}+\alpha _{2}\Delta _{23}\right) , \\
\eta _{2} &=&\left( 2\varepsilon \Delta _{11},\ \alpha _{1}\Delta _{11}+\alpha
_{2}\Delta _{12}+\frac{\alpha _{2}\left( \alpha _{2}-2\varepsilon \right) }{%
\alpha _{1}}\Delta _{22},\ \varepsilon \Delta _{13}\right) , \\
\eta _{3} &=&\left( 2\varepsilon \Delta _{11}+\ \alpha _{1}\Delta _{12},\alpha
_{1}\Delta _{22},\varepsilon \Delta _{13}+\alpha _{1}\ \Delta _{23}\right) , \\
\eta _{4} &=&\left( 2\varepsilon \Delta _{11},\ 2\varepsilon \Delta _{12}+\alpha
_{1}\Delta _{22},\varepsilon \Delta _{13}\right)  \\
\eta _{5} &=&\left( 2\varepsilon \Delta _{11},\alpha _{1}\Delta _{22},\varepsilon
\Delta _{13}\right) ,
\end{eqnarray*}%
for some $\alpha _{1},\alpha _{2}\in \mathbb{C}$ and $\varepsilon \in \{\pm 1\}$, where $\alpha _{2}\neq 0$ in $\eta _{1}$ and $\alpha _{1}\neq 0$ in $\eta _{2}$ (the forms with a vanishing denominator belong to the other families). The automorphism group $\mathrm{Aut}\left( \mathbf{A}_{07}\right) $
consists of the automorphisms $\phi $ given by a matrix of the following
forms: 
\begin{equation*}
\phi =%
\begin{pmatrix}
1 & 0 & 0 \\ 
a_{21} & a_{22} & 0 \\ 
0 & 0 & a_{33}%
\end{pmatrix}%
.
\end{equation*}%
We proceed to classify the orbits of $\eta _{i}$ under the action of $%
\mathrm{Aut}\left( \mathbf{A}_{07}\right) $, distinguishing the following
cases.

\begin{itemize}
\item $\theta =\eta _{1}$. We choose $\phi $ to be the following
automorphism: 
\begin{equation*}
\phi =%
\begin{pmatrix}
1 & 0 & 0 \\ 
-\frac{\alpha _{1}}{\alpha _{2}} & \frac{1}{\alpha _{2}} & 0 \\ 
0 & 0 & 1%
\end{pmatrix}%
.
\end{equation*}%
Then $\theta \ast \phi \ =\left( \Delta _{12},\ \Delta _{11}+\Delta _{22},\
\Delta _{23}\right) $. So we get the superalgebra $\mathbf{S}_{17}$.

\item $\theta =\eta _{2}$. If $\alpha _{2}=2\varepsilon $, we choose $\phi $ to
be the following automorphism:%
\begin{equation*}
\phi =%
\begin{pmatrix}
1 & 0 & 0 \\ 
-\frac{\alpha _{1}}{2\varepsilon } & 1 & 0 \\ 
0 & 0 & 1%
\end{pmatrix}%
.
\end{equation*}%
Then $\theta \ast \phi \ =\left( 2\varepsilon \Delta _{11},\ 2\varepsilon \Delta
_{12},\ \varepsilon \Delta _{13}\right) $. So we get the superalgebras $\mathbf{%
S}_{18},\mathbf{S}_{19}$. For $\alpha _{2}\neq 2\varepsilon $, we choose $\phi $
to be the first of the following matrices if $\alpha _{2}=0$, or the second one
if $\alpha _{2}\neq 0$: 
\begin{equation*}
\begin{pmatrix}
1 & 0 & 0 \\ 
-\frac{\alpha _{1}}{-2\varepsilon +\alpha _{2}} & 1 & 0 \\ 
0 & 0 & 1%
\end{pmatrix}%
,%
\begin{pmatrix}
1 & 0 & 0 \\ 
-\frac{\alpha _{1}}{-2\varepsilon +\alpha _{2}} & \frac{\alpha _{1}}{\alpha
_{2}\left( -2\varepsilon +\alpha _{2}\right) } & 0 \\ 
0 & 0 & 1%
\end{pmatrix}%
.
\end{equation*}%
Then $\theta \ast \phi \ \in \{\left( 2\varepsilon \Delta _{11},0,\ \varepsilon
\Delta _{13}\right) ,\left( 2\varepsilon \Delta _{11},\ \Delta _{22},\ \varepsilon
\Delta _{13}\right) \}$. This yields the superalgebras $\mathbf{S}_{20},%
\mathbf{S}_{21},\mathbf{S}_{22},\mathbf{S}_{23}$.

\item $\theta =\eta _{3}$. If $\alpha _{1}=0$, we obtain the superalgebras $%
\mathbf{S}_{20},\mathbf{S}_{21}$. If $\alpha _{1}\neq 0$, we choose $\phi $
to be the following automorphism:%
\begin{equation*}
\phi =%
\begin{pmatrix}
1 & 0 & 0 \\ 
-\frac{\varepsilon }{\alpha _{1}} & \frac{1}{\alpha _{1}} & 0 \\ 
0 & 0 & 1%
\end{pmatrix}%
.
\end{equation*}%
Then $\theta \ast \phi =\left(  \Delta _{12},\Delta
_{11}+\Delta
_{22}, \Delta _{23}\right) $. Consequently, we obtain
the superalgebra $\mathbf{S}_{17}$.

\item $\theta =\eta _{4}$. If $\alpha _{1}=0$, we obtain the superalgebras $%
\mathbf{S}_{18},\mathbf{S}_{19}$. If $\alpha _{1}\neq 0$, we choose $\phi $
to be the following automorphism:%
\begin{equation*}
\phi =%
\begin{pmatrix}
1 & 0 & 0 \\ 
-\frac{2\varepsilon }{\alpha _{1}} & \frac{1}{\alpha _{1}} & 0 \\ 
0 & 0 & 1%
\end{pmatrix}%
.
\end{equation*}%
Then $\theta \ast \phi =\left( 2\varepsilon \Delta _{11},\ \Delta
_{22},\varepsilon \Delta _{13}\right) $. Hence, we obtain the superalgebras $%
\mathbf{S}_{22},\mathbf{S}_{23}$.

\item $\theta =\eta _{5}$. If $\alpha _{1}=0$, we obtain the superalgebras $%
\mathbf{S}_{20},\mathbf{S}_{21}$. If $\alpha _{1}\neq 0$, we choose $\phi $
to be the following automorphism:%
\begin{equation*}
\phi =%
\begin{pmatrix}
1 & 0 & 0 \\ 
0 & \frac{1}{\alpha _{1}} & 0 \\ 
0 & 0 & 1%
\end{pmatrix}%
.
\end{equation*}%
Then $\theta \ast \phi =\left( 2\varepsilon \Delta _{11},\Delta _{22},\varepsilon
\Delta _{13}\right) $. So we obtain the superalgebras $\mathbf{S}_{22},%
\mathbf{S}_{23}$.
\end{itemize}

\item \underline{$\mathrm{S}^{-}=\mathbf{A}_{10}$}. Let $\theta =\big(%
\mathrm{B}_{1},\ \mathrm{B}_{2},\ \mathrm{B}_{3}\big)$ be an arbitrary
element of $\mathrm{Z}^{2}\left( \mathbf{A}_{10},\mathbf{A}_{10}\right) $. Then $%
\theta \in \left\{ \eta _{1},\eta _{2},\eta _{3}\right\} $ where%
\begin{eqnarray*}
\ \eta _{1} &=&\left( \alpha _{1}\Delta _{11}+\alpha _{2}\Delta _{12}+\alpha
_{3}\Delta _{22},\ \frac{\alpha _{2}^{2}-\alpha _{1}\alpha _{3}}{\alpha _{2}}%
\Delta _{22},\ \alpha _{1}\Delta _{13}+\alpha _{2}\Delta _{23}\right) , \\
\eta _{2} &=&\left( \alpha _{1}\Delta _{11},\ \alpha _{2}\Delta _{22},\
\alpha _{1}\Delta _{13}\right) , \\
\eta _{3} &=&\left( \alpha _{1}\Delta _{22},\ \alpha _{2}\Delta
_{22},0\right) ,
\end{eqnarray*}%
for some $\alpha _{1},\alpha _{2},\alpha _{3}\in \mathbb{C}$, where $\alpha _{2}\neq 0$ in $\eta _{1}$. The
automorphism group $\mathrm{Aut}\left( \mathbf{A}_{10}\right) $ consists of
the automorphisms $\phi $ given by a matrix of the following form:%
\begin{equation*}
\phi =%
\begin{pmatrix}
a_{11} & a_{12} & 0 \\ 
0 & a_{22} & 0 \\ 
0 & 0 & \sigma \sqrt{a_{11}}%
\end{pmatrix}%
:\sigma =\pm 1.
\end{equation*}%
The trivial case $\theta =0$ yields the superalgebra $\mathbf{S}_{24}$. For $%
\theta \neq 0$, we distinguish the following cases:

\begin{itemize}
\item $\theta =\eta _{1}$. If $\alpha _{1}=0$, we choose $\phi $ to be the
following automorphism:%
\begin{equation*}
\phi =%
\begin{pmatrix}
1 & -\frac{\alpha _{3}}{\alpha _{2}^{2}} & 0 \\ 
0 & \frac{1}{\alpha _{2}} & 0 \\ 
0 & 0 & 1%
\end{pmatrix}%
.
\end{equation*}%
Then $\theta \ast \phi =\left( \Delta _{12},\ \Delta _{22},\ \Delta
_{23}\right) $. Hence we get the superalgebra $\mathbf{S}_{25}$. For $\alpha
_{1}\neq 0$, we choose $\phi $ to be the first of the following matrices if $%
\alpha _{2}^{2}-\alpha _{1}\alpha _{3}=0$, or the second if $\alpha
_{2}^{2}-\alpha _{1}\alpha _{3}\neq 0$:%
\begin{equation*}
\begin{pmatrix}
\frac{1}{\alpha _{1}} & -\frac{\alpha _{2}}{\alpha _{1}} & 0 \\ 
0 & 1 & 0 \\ 
0 & 0 & \sqrt{\frac{1}{\alpha _{1}}}%
\end{pmatrix}%
,%
\begin{pmatrix}
\frac{1}{\alpha _{1}} & \frac{\alpha _{2}^{2}}{\alpha _{1}\left( \alpha
_{1}\alpha _{3}-\alpha _{2}^{2}\right) } & 0 \\ 
0 & \frac{\alpha _{2}}{\alpha _{2}^{2}-\alpha _{1}\alpha _{3}} & 0 \\ 
0 & 0 & \sqrt{\frac{1}{\alpha _{1}}}%
\end{pmatrix}%
.
\end{equation*}%
Then $\theta \ast \phi \ \in \left\{ \left( \Delta _{11},\ 0,\ \Delta
_{13}\right) ,\left( \Delta _{11},\ \Delta _{22},\ \Delta _{13}\right)
\right\} $. Thus we obtain the superalgebras $\mathbf{S}_{26},\mathbf{S}_{27}
$.

\item $\theta =\eta _{2}$. If $\alpha _{1}\alpha _{2}\neq 0$, we choose $%
\phi $ to be the following automorphism:%
\begin{equation*}
\phi =%
\begin{pmatrix}
\frac{1}{\alpha _{1}} & 0 & 0 \\ 
0 & \frac{1}{\alpha _{2}} & 0 \\ 
0 & 0 & \sqrt{\frac{1}{\alpha _{1}}}%
\end{pmatrix}%
.
\end{equation*}%
Then $\theta \ast \phi =\left( \Delta _{11},\ \Delta _{22},\ \Delta
_{13}\right) $. So we obtain the superalgebra $\mathbf{S}_{27}$. For $\alpha
_{1}\alpha _{2}=0$, we choose $\phi $ to be the first of the following
matrices if $\alpha _{1}\neq 0$, or the second if $\alpha _{1}=0$ (then $\alpha _{2}\neq 0$, since $\theta\neq 0$):%
\begin{equation*}
\begin{pmatrix}
\frac{1}{\alpha _{1}} & 0 & 0 \\ 
0 & 1 & 0 \\ 
0 & 0 & \sqrt{\frac{1}{\alpha _{1}}}%
\end{pmatrix}%
,%
\begin{pmatrix}
1 & 0 & 0 \\ 
0 & \frac{1}{\alpha _{2}} & 0 \\ 
0 & 0 & 1%
\end{pmatrix}%
.
\end{equation*}%
Then $\theta \ast \phi \ \in \left\{ \left( \Delta _{11},\ 0,\Delta
_{13}\right) ,\left( 0,\ \Delta _{22},\ 0\right) \right\} $. So we obtain
the superalgebras $\mathbf{S}_{26},\mathbf{S}_{28}$.

\item $\theta =\eta _{3}$. If $\alpha _{1}\alpha _{2}\neq 0$, we choose $%
\phi $ to be the following automorphism:%
\begin{equation*}
\phi =%
\begin{pmatrix}
1 & \frac{\alpha _{1}}{\alpha _{2}^{2}} & 0 \\ 
0 & \frac{1}{\alpha _{2}} & 0 \\ 
0 & 0 & 1%
\end{pmatrix}%
.
\end{equation*}%
Then $\theta \ast \phi =\left(0,\ \Delta _{22},\ 0\right) $. So
we obtain the superalgebra $\mathbf{S}_{28}$. For $\alpha _{1}\alpha _{2}=0$%
, we choose $\phi $ to be the first of the following matrices if $\alpha
_{1}\neq 0$, or the second if $\alpha _{1}=0$ (then $\alpha _{2}\neq 0$, since $\theta\neq 0$):%
\begin{equation*}
\begin{pmatrix}
\alpha _{1} & 0 & 0 \\ 
0 & 1 & 0 \\ 
0 & 0 & \sqrt{\alpha _{1}}%
\end{pmatrix}%
,%
\begin{pmatrix}
1 & 0 & 0 \\ 
0 & \frac{1}{\alpha _{2}} & 0 \\ 
0 & 0 & 1%
\end{pmatrix}%
.
\end{equation*}%
Then $\theta \ast \phi \ \in \left\{ \left( \Delta _{22},\ 0,0\right)
,\left( 0,\ \Delta _{22},\ 0\right) \right\} $. So we obtain the
superalgebras $\mathbf{S}_{29},\mathbf{S}_{28}$.
\end{itemize}

\end{enumerate}

\subsection{Corollaries}\label{sec:cor}

\begin{definition}
A superalgebra $\mathrm{A}$ is called \emph{assosymmetric} if for all $x,y,z\in\mathrm{A}$
\begin{equation}
    (x,y,z)=(-1)^{\bar{x}\bar{y}}(y,x,z)=(-1)^{\bar{y}\bar{z}}(x,z,y). \label{asso}
\end{equation}
\end{definition}

\begin{lemma}\label{lem:assosym}
A superalgebra $A$ over a field of characteristic different from $2$ and $3$ is assosymmetric if and only if it is semi-alternative and Lie-admissible.
\end{lemma}

\begin{proof}
    Let $J(x,y,z)=(-1)^{\bar x\bar z}[[x,y],z]+(-1)^{\bar x\bar y}[[y,z],x]+(-1)^{\bar y\bar z}[[z,x],y]$ be the super-Jacobiator of $A^-$; to avoid the factor $\frac{1}{4}$ we compute it with the unnormalised supercommutator $[x,y]=xy-(-1)^{\bar x\bar y}yx$, which does not affect the super-Jacobi identity. A direct calculation gives
\begin{align}
\begin{aligned}
J(x,y,z)
={}&
(-1)^{\bar x\bar z}(xy)z
-(-1)^{\bar x\bar z+\bar x\bar y}(yx)z
-(-1)^{\bar y\bar z}z(xy)
+(-1)^{\bar x\bar y+\bar y\bar z}z(yx)\\
&+(-1)^{\bar x\bar y}(yz)x
-(-1)^{\bar x\bar y+\bar y\bar z}(zy)x
-(-1)^{\bar x\bar z}x(yz)
+(-1)^{\bar x\bar z+\bar y\bar z}x(zy)\\
&+(-1)^{\bar y\bar z}(zx)y
-(-1)^{\bar y\bar z+\bar x\bar z}(xz)y
-(-1)^{\bar x\bar y}y(zx)
+(-1)^{\bar x\bar z+\bar x\bar y}y(xz).
\end{aligned}
\label{expansion}
\end{align}

Grouping terms by associators yields
\begin{align}
\begin{aligned}
J(x,y,z)
={}&
(-1)^{\bar x\bar z}(x,y,z)
-(-1)^{\bar x\bar z+\bar x\bar y}(y,x,z)
+(-1)^{\bar x\bar y}(y,z,x)\\
&-(-1)^{\bar x\bar y+\bar y\bar z}(z,y,x)
+(-1)^{\bar y\bar z}(z,x,y)
-(-1)^{\bar y\bar z+\bar x\bar z}(x,z,y).
\end{aligned}
\label{associator}
\end{align}

\paragraph{($\Rightarrow$)}
Assume $A$ is assosymmetric. Then from \eqref{asso} we have
\[
(x,y,z)=(-1)^{\bar x\bar y}(y,x,z), \qquad (y,z,x)=(-1)^{\bar y\bar z}(z,y,x), \qquad (z,x,y)=(-1)^{\bar x\bar z}(x,z,y).
\]
Substituting these relations into \eqref{associator}, we obtain
\(
J(x,y,z)=0
\) for all $x,y,z\in A$. Hence the supercommutator $A^-$ satisfies the super-Jacobi identity, and therefore $A$ is Lie-admissible. Moreover, applying the two identities of \eqref{asso} successively, we get
\[
(x,y,z)=(-1)^{\bar x\bar y}(y,x,z)=(-1)^{\bar x\bar y+\bar x\bar z}(y,z,x),
\]
i.e. $A$ is semi-alternative. Note that this implication holds over any field, and the converse implication uses only the invertibility of $3$.
\paragraph{($\Leftarrow$)}
Conversely, suppose $A$ is semi-alternative and Lie-admissible. From \eqref{semi-alt}, we obtain the following relations:
\begin{align*}
(x,y,z)
&=(-1)^{\bar x\bar y+\bar x\bar z}(y,z,x),
&
(z,x,y)
&=(-1)^{\bar x\bar z+\bar y\bar z}(x,y,z),
\\
(y,x,z)
&=(-1)^{\bar x\bar y+\bar y\bar z}(x,z,y),
&
(z,y,x)
&=(-1)^{\bar x\bar z+\bar y\bar z}(y,x,z).
\end{align*}
Since $A$ is Lie-admissible, we have
\(
J(x,y,z)=0.
\)
Using \eqref{associator} and substituting the above relations, we obtain
\begin{align*}
\begin{aligned}
0
={}&
(-1)^{\bar x\bar z}(x,y,z)
-(-1)^{\bar x\bar z+\bar x\bar y}(y,x,z)
+(-1)^{\bar x\bar y}
(-1)^{\bar x\bar y+\bar x\bar z}(x,y,z)\\
&\quad
-(-1)^{\bar x\bar y+\bar y\bar z}
(-1)^{\bar x\bar z+\bar y\bar z}(y,x,z)
+(-1)^{\bar y\bar z}
(-1)^{\bar x\bar z+\bar y\bar z}(x,y,z)
-(-1)^{\bar y\bar z+\bar x\bar z}
(-1)^{\bar x\bar y+\bar y\bar z}(y,x,z).
\end{aligned}
\end{align*}
Whence \begin{align*}
0
=
3(-1)^{\bar x\bar z}
\left(
(x,y,z)-(-1)^{\bar x\bar y}(y,x,z)
\right).
\end{align*}
Thus we obtain
\begin{equation}
(x,y,z)
=
(-1)^{\bar x\bar y}(y,x,z).
\label{asso-first}
\end{equation}
Combining the identity \((y,x,z)
=(-1)^{\bar x\bar y+\bar y\bar z}(x,z,y)\) with \eqref{asso-first} gives 
\begin{equation}
(x,y,z)=(-1)^{\bar y\bar z}(x,z,y).
\label{asso-second}
\end{equation}
Therefore,
\[
(x,y,z)
=
(-1)^{\bar x\bar y}(y,x,z)
=
(-1)^{\bar y\bar z}(x,z,y),
\]
which is precisely the assosymmetric identity \eqref{asso}. Hence, $A$ is assosymmetric.
\end{proof}

By the above lemma, every assosymmetric superalgebra is semi-alternative, and its supercommutator is a Lie superalgebra. Consequently, in view of Remarks~\ref{remark:Malcev-12-nonLie} and
\ref{remark:Malcev-21-nonLie} and the classifications given in
Theorems~A1 and~A2, we obtain the following corollaries.

\begin{corollary}\label{cor:asy12}
Let $A$ be a complex $3$-dimensional assosymmetric superalgebra of type $(1,2)$. Then $A$ is isomorphic to one of the superalgebras listed in Proposition~\ref{AC(1,2)} and Theorem~A1, except ${\rm S}_{06}$, ${\rm S}_{07}$, and ${\rm S}_{09}^{\alpha}$.
\end{corollary}

\begin{corollary}\label{cor:asy21}
Let $A$ be a complex $3$-dimensional assosymmetric superalgebra of type $(2,1)$. Then $A$ is isomorphic to one of the superalgebras listed in Proposition~\ref{AC(2,1)} and Theorem~A2, except $\mathbf S_{01}$, $\mathbf S_{02}$, $\mathbf S_{03}$, and $\mathbf S_{04}$.
\end{corollary}

Every associative superalgebra is semi-alternative. Checking the associativity of the superalgebras of Theorems~A1 and~A2 directly, we obtain the following corollaries.

\begin{corollary}\label{cor:as12}
Let $A$ be a complex $3$-dimensional associative superalgebra of type $(1,2)$. Then $A$ is isomorphic to one of the superalgebras listed in Proposition~\ref{AC(1,2)} and Theorem~A1, except ${\rm S}_{06}$, ${\rm S}_{07}$, and ${\rm S}_{09}^{\alpha}$. In particular, the associative and the assosymmetric superalgebras of type $(1,2)$ coincide.
\end{corollary}

\begin{corollary}\label{cor:as21}
Let $A$ be a complex $3$-dimensional associative superalgebra of type $(2,1)$. Then $A$ is isomorphic to one of the superalgebras listed in Proposition~\ref{AC(2,1)} and Theorem~A2, except
$\mathbf S_{01}$, $\mathbf S_{02}$, $\mathbf S_{03}$, $\mathbf S_{04}$, $\mathbf S_{07}$, $\mathbf S_{08}$, $\mathbf S_{13}$, $\mathbf S_{14}$, $\mathbf S_{15}$, $\mathbf S_{16}$, $\mathbf S_{32}$, and $\mathbf S_{33}$.
\end{corollary}

\begin{definition}
An associative superalgebra $A$ is called a \emph{$\mathfrak{perm}$ superalgebra} if for all $x,y,z\in A$
\begin{equation*}
(xy)z=(-1)^{\bar{x}\bar{y}}(yx)z.
\end{equation*}
\end{definition}

Every associative supercommutative superalgebra is a $\mathfrak{perm}$ superalgebra. Checking the $\mathfrak{perm}$ identity for the superalgebras of Corollaries~\ref{cor:as12} and~\ref{cor:as21}, we obtain the following corollaries.

\begin{corollary}\label{cor:perm12}
Let $A$ be a complex $3$-dimensional $\mathfrak{perm}$ superalgebra of type $(1,2)$. Then $A$ is isomorphic to one of the superalgebras listed in Proposition~\ref{AC(1,2)}, or to one of the superalgebras
${\rm S}_{01}^{\alpha}$, ${\rm S}_{02}$, ${\rm S}_{05}$, ${\rm S}_{08}^{\alpha}$, ${\rm S}_{11}$, ${\rm S}_{13}$, and ${\rm S}_{15}$ of Theorem~A1.
\end{corollary}

\begin{corollary}\label{cor:perm21}
Let $A$ be a complex $3$-dimensional $\mathfrak{perm}$ superalgebra of type $(2,1)$. Then $A$ is isomorphic to one of the superalgebras listed in Proposition~\ref{AC(2,1)}, or to one of the superalgebras
$\mathbf S_{06}$, $\mathbf S_{11}$, $\mathbf S_{12}$, $\mathbf S_{18}$, $\mathbf S_{20}$, $\mathbf S_{22}$, $\mathbf S_{24}$, $\mathbf S_{28}$, and $\mathbf S_{29}$ of Theorem~A2.
\end{corollary}

\section{The geometric classification of superalgebras}
\subsection{Preliminaries}\label{sec:method}

All spaces in this section are complex.
Let $\mathbb{V}=\mathbb{V}_{\bar0}\oplus\mathbb{V}_{\bar1}$ be a $\mathbb{Z}_2$-graded vector space of type $(n,m)$, i.e. ${\rm dim}\,\mathbb{V}_{\bar0}=n$ and ${\rm dim}\,\mathbb{V}_{\bar1}=m$, and let us fix a basis $e_1,\dots,e_n,f_1,\dots,f_m$ of $\mathbb{V}$; we also write $x_1,\dots,x_{n+m}$ for this basis, where $x_i=e_i$ for $i\le n$ and $x_{n+j}=f_j$ for $j\le m$.
A superalgebra structure on $\mathbb{V}$ is an even bilinear map $\mu\colon\mathbb{V}\times\mathbb{V}\to\mathbb{V}$, i.e. $\mu(\mathbb{V}_{\bar i},\mathbb{V}_{\bar j})\subseteq\mathbb{V}_{\bar i+\bar j}$. It is determined by its structure constants $c_{ij}^k\in\mathbb{C}$ such that $\mu(x_i,x_j)=\sum\limits_{k=1}^{n+m}c_{ij}^kx_k$, where $c_{ij}^k=0$ whenever $\bar x_i+\bar x_j\neq\bar x_k$. Hence the set of all superalgebra structures on $\mathbb{V}$ is an affine space, and a subset of it is {\it Zariski-closed} if it can be defined by a set of polynomial equations in the variables $c_{ij}^k$.
If $\mathfrak{V}$ is a variety of superalgebras defined by a family of polynomial super-identities, then the structures from $\mathfrak{V}$ form a Zariski-closed subset $\mathfrak{V}_{(n,m)}$.

The group $G={\rm GL}(\mathbb{V}_{\bar0})\times{\rm GL}(\mathbb{V}_{\bar1})$ of even invertible linear maps of $\mathbb{V}$, which has dimension $n^2+m^2$, acts on $\mathfrak{V}_{(n,m)}$ by conjugations:
$$(g*\mu)(x,y)=g\mu\big(g^{-1}x,g^{-1}y\big)$$
for $x,y\in\mathbb{V}$, $\mu\in\mathfrak{V}_{(n,m)}$ and $g\in G$.
Thus, $\mathfrak{V}_{(n,m)}$ is decomposed into $G$-orbits that correspond to the isomorphism classes of superalgebras; the algebraic classification of Section~1 is the description of these orbits.
Let ${\mathcal O}(\mu)$ denote the orbit of $\mu$ and $\overline{{\mathcal O}(\mu)}$ denote its Zariski closure. In what follows we identify a superalgebra with the corresponding structure, and we denote by $\mathbb{C}^{n+m}$ the superalgebra with zero multiplication.

Let $A$ and $B$ be two superalgebras of type $(n,m)$ represented by the structures $\mu$ and $\lambda$ from $\mathfrak{V}_{(n,m)}$, respectively.
We say that $A$ degenerates to $B$ and write $A\to B$ if $\lambda\in\overline{{\mathcal O}(\mu)}$. Note that in this case we have $\overline{{\mathcal O}(\lambda)}\subseteq\overline{{\mathcal O}(\mu)}$; hence, the definition of a degeneration does not depend on the choice of $\mu$ and $\lambda$.
If $A\not\cong B$, then the assertion $A\to B$ is called a {\it proper degeneration}. We write $A\not\to B$ if $\lambda\not\in\overline{{\mathcal O}(\mu)}$.

The superalgebra $A$ is {\it rigid} in $\mathfrak{V}_{(n,m)}$ if ${\mathcal O}(\mu)$ is an open subset of $\mathfrak{V}_{(n,m)}$.
Recall that a subset of a variety is called irreducible if it cannot be represented as a union of two proper closed subsets, and a maximal irreducible closed subset of a variety is called an {\it irreducible component}. In particular, $A$ is rigid in $\mathfrak{V}_{(n,m)}$ iff $\overline{{\mathcal O}(\mu)}$ is an irreducible component of $\mathfrak{V}_{(n,m)}$.
It is well known that any affine variety can be represented as a finite union of its irreducible components in a unique way. We denote by ${\rm Rig}(\mathfrak{V}_{(n,m)})$ the set of rigid superalgebras in $\mathfrak{V}_{(n,m)}$.
We denote by $\mathfrak{Der}_{\bar0}(A)$ the space of even derivations of $A$, and given spaces $\mathbb{U}$ and $\mathbb{W}$, we write simply $\mathbb{U}>\mathbb{W}$ instead of ${\rm dim}\,\mathbb{U}>{\rm dim}\,\mathbb{W}$.

\begin{lemma}\label{lem_der}
We have ${\rm dim}\,{\mathcal O}(\mu)=n^2+m^2-{\rm dim}\,\mathfrak{Der}_{\bar0}(\mu)$. In particular, if $A\to B$ is a proper degeneration, then $\mathfrak{Der}_{\bar0}(A)<\mathfrak{Der}_{\bar0}(B)$.
\end{lemma}
\begin{proof}
The stabilizer of $\mu$ in $G$ is the group of even automorphisms of the superalgebra, whose Lie algebra is $\mathfrak{Der}_{\bar0}(\mu)$; hence ${\rm dim}\,{\mathcal O}(\mu)=n^2+m^2-{\rm dim}\,\mathfrak{Der}_{\bar0}(\mu)$. The orbit ${\mathcal O}(\mu)$ is open in its closure, and the boundary $\overline{{\mathcal O}(\mu)}\setminus{\mathcal O}(\mu)$ is a closed $G$-stable set of smaller dimension; thus ${\rm dim}\,{\mathcal O}(\lambda)<{\rm dim}\,{\mathcal O}(\mu)$ for $\lambda\in\overline{{\mathcal O}(\mu)}\setminus{\mathcal O}(\mu)$.
\end{proof}

By Lemma~\ref{lem_der}, an orbit closure does not contain other orbits of the same or larger dimension. If $A\to C$ and $C\to B$, then $A\to B$. It is easy to see that any superalgebra degenerates to the superalgebra with zero multiplication.
In all the cases considered below, $\mathfrak{V}_{(n,m)}$ is a finite union of orbits and of one-parameter families of orbits. Hence $\mathfrak{V}_{(n,m)}$ is the union of the closures of the orbits of some superalgebras $A_1,\dots,A_p$ and of the closures of the unions of the orbits of some families $\mathcal{F}_1,\dots,\mathcal{F}_q$; these closures are irreducible, and they are exactly the irreducible components of $\mathfrak{V}_{(n,m)}$ if none of them is contained in another one. Thus, to obtain the geometric classification it is enough (a) to show that every superalgebra of $\mathfrak{V}_{(n,m)}$ belongs to one of these closures, for which we give a degeneration (or a chain of degenerations) from one of the $A_i$ or $\mathcal{F}_j$, and (b) to prove that none of these closures is contained in another one, i.e. to prove the non-degenerations between the $A_i$ and $\mathcal{F}_j$.

\medskip

{\it Degenerations.} To prove degenerations, we construct families of matrices parametrized by $t$. Namely, let $A$ and $B$ be two superalgebras represented by the structures $\mu$ and $\lambda$, and let $c_{ij}^k$ be the structure constants of $\lambda$ in the basis $e_1,\dots,e_n,f_1,\dots,f_m$. If there exist even vectors ${\rm E}^t_1,\dots,{\rm E}^t_n\in\mathbb{V}_{\bar0}$ and odd vectors ${\rm F}^t_1,\dots,{\rm F}^t_m\in\mathbb{V}_{\bar1}$ ($t\in\mathbb{C}^*$) whose coordinates are functions of $t$, which form a basis of $\mathbb{V}$ for any small $t\neq0$, and such that the structure constants $c_{ij}^k(t)$ of $\mu$ in the basis ${\rm E}^t_1,\dots,{\rm E}^t_n,{\rm F}^t_1,\dots,{\rm F}^t_m$ satisfy $c_{ij}^k(t)\to c_{ij}^k$ as $t\to0$, then $A\to B$. In this case ${\rm E}^t_1,\dots,{\rm F}^t_m$ is called a {\it parametrized basis} for $A\to B$.
In the tables of degenerations the vectors ${\rm E}^t_i$, ${\rm F}^t_j$ are written in the basis $e_1,\dots,e_n,f_1,\dots,f_m$ of the superalgebra $A$ from Theorems~A1 and~A2.

{\it Families of superalgebras.} The varieties under consideration contain one-parameter families of superalgebras $\{\mu^{\alpha}\}_{\alpha\in\mathbb{C}}$. We write $\{\mu^{\alpha}\}_{\alpha\in\mathbb{C}}\to\lambda$ if $\lambda\in\overline{\bigcup_\alpha{\mathcal O}(\mu^{\alpha})}$; for a family target, $\mu\to\{\lambda^{\beta}\}_{\beta\in\mathbb{C}}$ means that $\mu\to\lambda^{\beta}$ for every $\beta$. A parametrized basis for $\{\mu^{\alpha}\}_{\alpha\in\mathbb{C}}\to\lambda$ is a parametrized basis ${\rm E}^t_1,\dots,{\rm F}^t_m$ together with a member $\mu^{\alpha(t)}$ of the family depending on $t$ such that the structure constants of $\mu^{\alpha(t)}$ in the basis ${\rm E}^t_1,\dots,{\rm F}^t_m$ tend to the structure constants of $\lambda$; then $\{\mu^{\alpha}\}_{\alpha\in\mathbb{C}}\to\lambda$. Note that $\{\mu^{\alpha}\}_{\alpha\in\mathbb{C}}\to\lambda$ does not imply that some member $\mu^{\alpha}$ degenerates to $\lambda$: if ${\rm dim}\,{\mathcal O}(\mu^{\alpha})\le{\rm dim}\,{\mathcal O}(\lambda)$ for all $\alpha$ and $\lambda$ is not isomorphic to a member of the family, then no single member $\mu^{\alpha}$ degenerates to $\lambda$ by Lemma~\ref{lem_der}, although $\{\mu^{\alpha}\}_{\alpha\in\mathbb{C}}\to\lambda$ may hold.

\medskip

{\it Non-degenerations.} The main tools are the following two lemmas.

\begin{lemma}\label{main}
Let $\mathcal{B}$ be a Borel subgroup of $G$ and $\mathcal{R}\subseteq\mathfrak{V}_{(n,m)}$ be a closed subset.
If $A\to B$ and $A$ can be represented by a structure $\mu$ such that $b*\mu\in\mathcal{R}$ for all $b\in\mathcal{B}$ (for example, $\mu\in\mathcal{R}$ and $\mathcal{R}$ is $\mathcal{B}$-stable), then there is $\lambda\in\mathcal{R}$ that represents $B$. The same holds for a family $A=\{\mu^{\alpha}\}_{\alpha\in\mathbb{C}}$ (i.e. if $\lambda\in\overline{\bigcup_{\alpha}{\mathcal O}(\mu^{\alpha})}$) if every member of the family can be represented by such a structure.
\end{lemma}
\begin{proof}
The set $\mathcal{R}'=\{\nu\mid b*\nu\in\mathcal{R}\mbox{ for all }b\in\mathcal{B}\}=\bigcap_{b\in\mathcal{B}}b^{-1}*\mathcal{R}$ is closed, $\mathcal{B}$-stable, contains $\mu$ and is contained in $\mathcal{R}$. The set $G*\mathcal{R}'$ is the image of the closed subset $\big\{(g\mathcal{B},\nu)\mid g^{-1}*\nu\in\mathcal{R}'\big\}$ of $G/\mathcal{B}\times\mathfrak{V}_{(n,m)}$ under the projection to the second factor. Since the flag variety $G/\mathcal{B}$ is complete, this projection is closed, so $G*\mathcal{R}'$ is a closed set containing ${\mathcal O}(\mu)$; hence it contains $\overline{{\mathcal O}(\mu)}\ni\lambda'$ for every structure $\lambda'$ representing $B$, i.e. $g*\lambda'\in\mathcal{R}'\subseteq\mathcal{R}$ for some $g\in G$.
\end{proof}

{\it Borel subgroups and the sets $\mathcal{R}$.} Let $x_1,x_2,x_3$ be a basis of a superalgebra $A$ of dimension $3$: for type $(1,2)$, the vector $x_1$ is even and $x_2,x_3$ are odd; for type $(2,1)$, the vectors $x_1,x_2$ are even and $x_3$ is odd. We denote by $c_{ij}^k$ the structure constants of $A$ in this basis, $x_ix_j=\sum_{k=1}^{3}c_{ij}^kx_k$, put $A_i=\langle x_i,\dots,x_3\rangle$ for $i=1,2,3$ and $A_4=0$, and write $A=\langle x_1,x_2,x_3\rangle$. The condition $A_iA_j\subseteq A_k$ means that $c_{i'j'}^{l}=0$ for all $i'\ge i$, $j'\ge j$ and $l<k$; in particular, $A_iA_j=0$ means $A_iA_j\subseteq A_4$. As the Borel subgroup $\mathcal{B}$ of $G$ we always take the group of upper triangular changes of basis from $G$
$$x_i\mapsto \sum_{j\ge i}b_{ij}x_j \qquad (b_{ii}\neq0,\ b_{ij}=0 \mbox{ if } x_i \mbox{ and } x_j \mbox{ have different parities}),$$
i.e. $x_1\mapsto b_{11}x_1$, $x_2\mapsto b_{22}x_2+b_{23}x_3$, $x_3\mapsto b_{33}x_3$ for type $(1,2)$, and $x_1\mapsto b_{11}x_1+b_{12}x_2$, $x_2\mapsto b_{22}x_2$, $x_3\mapsto b_{33}x_3$ for type $(2,1)$.

All the sets used in the proofs of non-degenerations have the form
$$\mathcal{R}=\big\{A\mid A=\langle x_1,x_2,x_3\rangle,\ \mbox{conditions}\big\},$$
where the conditions are a system of linear equations in the structure constants $c_{ij}^k$: conditions of the form $A_iA_j\subseteq A_k$ (or $c_{ij}^k=0$) and linear relations between the structure constants. In each case the set $\mathcal{Z}$ of the structures satisfying the conditions is $\mathcal{B}$-stable, i.e. if the structure constants of $A$ in a basis $x_1,x_2,x_3$ satisfy the conditions, then the structure constants of $A$ in every basis obtained from $x_1,x_2,x_3$ by a change of basis from $\mathcal{B}$ satisfy them too. We write $A\in\mathcal{R}$ if $A$ has a basis $x_1,x_2,x_3$ satisfying the conditions (that is, $A$ is represented by a structure from $\mathcal{Z}$), and $B\cap\mathcal{R}=\varnothing$ if $B$ has no such basis. By Lemma~\ref{main} (applied to $\mathcal{Z}$), if $A\in\mathcal{R}$ and $B\cap\mathcal{R}=\varnothing$, then $A\not\to B$; if every member $A^{\alpha}$ of a family belongs to $\mathcal{R}$, then $\{A^{\alpha}\}_{\alpha\in\mathbb{C}}\not\to B$.

Since $\mathcal{Z}$ is $\mathcal{B}$-stable, to prove $B\cap\mathcal{R}=\varnothing$ it is enough to check representatives of the cosets $\mathcal{B}\backslash G$: if $e_1,f_1,f_2$ (resp. $e_1,e_2,f_1$) is the basis of $B$ from Theorem~A1 (resp. Theorem~A2), then it suffices to check that neither the basis $x_1=e_1$, $x_2=f_1$, $x_3=f_2$ nor the bases $x_1=e_1$, $x_2=f_2$, $x_3=f_1+sf_2$, $s\in\mathbb{C}$ (resp. $x_1=e_1$, $x_2=e_2$, $x_3=f_1$ and $x_1=e_2$, $x_2=e_1+se_2$, $x_3=f_1$, $s\in\mathbb{C}$) satisfy the conditions. For a family $B=\{B^{\beta}\}$ this gives a system of polynomial equations in $s$ and $\beta$, which determines the members of the family having a basis satisfying the conditions. The $\mathcal{B}$-stability of $\mathcal{Z}$ and the absence of solutions of these systems were verified by computer (the solution sets were computed by Gr\"obner bases).

{\it The proofs of non-degenerations.} Let $X$ and $Y$ be two different sets from the list in (b) above, i.e. closures of orbits of non-isomorphic superalgebras or closures of unions of orbits of families, and suppose $Y\subseteq X$. Both sets are irreducible, so ${\rm dim}\,Y\le{\rm dim}\,X$. If ${\rm dim}\,Y={\rm dim}\,X$, then $X=Y$; the orbit (resp. the union of the orbits of the family) defining $X$ and the one defining $Y$ are dense constructible subsets of $X$, hence they intersect, which is impossible since the corresponding superalgebras are not isomorphic. Therefore it is enough to prove $A\not\to B$ only for the pairs with ${\rm dim}\,A>{\rm dim}\,B$, where the dimension of a family is the dimension of the closure of the union of its orbits. In the tables of non-degenerations they are grouped by the first superalgebra: each row lists a superalgebra (or a family) $A$ and superalgebras $B$ such that $A\not\to B$, and gives a set $\mathcal{R}$ as above together with a basis $x_1,x_2,x_3$ of $A$ (for a family, of every member of the family) satisfying the conditions; the orbits of all the listed $B$ do not meet $\mathcal{R}$, so $A\not\to B$ by Lemma~\ref{main}.

\medskip

{\it Subvarieties.} If $\mathfrak{W}\subseteq\mathfrak{V}$ is a subvariety, then $\mathfrak{W}_{(n,m)}$ is a closed $G$-stable subset of $\mathfrak{V}_{(n,m)}$; in particular, $\overline{{\mathcal O}(\mu)}\subseteq\mathfrak{W}_{(n,m)}$ for every $\mu\in\mathfrak{W}_{(n,m)}$.
Hence the degenerations in $\mathfrak{W}_{(n,m)}$ are exactly the degenerations in $\mathfrak{V}_{(n,m)}$ between superalgebras from $\mathfrak{W}$. We have the inclusions
\[
\begin{array}{c}
\{\text{associative supercommutative}\}\subseteq\{\mathfrak{perm}\}\subseteq\{\text{associative}\}\\[2pt]
\subseteq\{\text{assosymmetric}\}\subseteq\{\text{semi-alternative}\},
\end{array}
\]
so all the degenerations and non-degenerations used for the subvarieties are degenerations and non-degenerations of $3$-dimensional semi-alternative superalgebras (Sections~\ref{sec:deg12}, \ref{sec:deg21} and~\ref{sec:subvar}).
We denote by $\mathfrak{SA}_{(n,m)}$, $\mathfrak{AsSym}_{(n,m)}$, $\mathfrak{As}_{(n,m)}$, $\mathfrak{Perm}_{(n,m)}$ and $\mathfrak{ACom}_{(n,m)}$ the varieties of semi-alternative, assosymmetric, associative, $\mathfrak{perm}$ and associative supercommutative superalgebras of type $(n,m)$, respectively.

\subsection{The geometric classification of semi-alternative superalgebras of type \texorpdfstring{$(1,2)$}{(1,2)}}\label{sec:deg12}
The semi-alternative superalgebras of type $(1,2)$ are listed in Theorem~A1 (and Proposition~\ref{AC(1,2)}); here ${\rm S}_{01}^{\alpha}$, ${\rm S}_{08}^{\alpha}$ and ${\rm S}_{09}^{\alpha}$ are one-parameter families, and $\{{\rm S}_{01}^{\alpha}\}_{\alpha\in\mathbb{C}}$, $\{{\rm S}_{08}^{\alpha}\}_{\alpha\in\mathbb{C}}$, $\{{\rm S}_{09}^{\alpha}\}_{\alpha\in\mathbb{C}}$ denote the families themselves. Recall that ${\rm S}_{01}^{\alpha}\cong{\rm S}_{01}^{-\alpha}$.

\begin{theoremB1}\label{theoremB1}
The variety $\mathfrak{SA}_{(1,2)}$ has dimension $4$ and $14$ irreducible components:
$$
\begin{array}{@{}l@{\quad}l@{\quad}l@{\quad}l@{\quad}l@{}}
\mathcal{C}_{01}=\overline{{\mathcal O}({\rm R}_{19})}; & \mathcal{C}_{04}=\overline{{\mathcal O}({\rm S}_{03})}; & \mathcal{C}_{07}=\overline{\bigcup_{\alpha}{\mathcal O}({\rm S}_{09}^{\alpha})}; & \mathcal{C}_{10}=\overline{{\mathcal O}({\rm S}_{12})}; & \mathcal{C}_{13}=\overline{{\mathcal O}({\rm S}_{15})};\\[4pt]
\mathcal{C}_{02}=\overline{{\mathcal O}({\rm R}_{24})}; & \mathcal{C}_{05}=\overline{{\mathcal O}({\rm S}_{06})}; & \mathcal{C}_{08}=\overline{{\mathcal O}({\rm S}_{10})}; & \mathcal{C}_{11}=\overline{{\mathcal O}({\rm S}_{13})}; & \mathcal{C}_{14}=\overline{{\mathcal O}({\rm S}_{16})}.\\[4pt]
\mathcal{C}_{03}=\overline{\bigcup_{\alpha}{\mathcal O}({\rm S}_{01}^{\alpha})}; & \mathcal{C}_{06}=\overline{{\mathcal O}({\rm S}_{07})}; & \mathcal{C}_{09}=\overline{{\mathcal O}({\rm S}_{11})}; & \mathcal{C}_{12}=\overline{{\mathcal O}({\rm S}_{14})};
\end{array}
$$
${\rm Rig}(\mathfrak{SA}_{(1,2)})=\big\{{\rm R}_{19},\allowbreak\ {\rm R}_{24},\allowbreak\ {\rm S}_{03},\allowbreak\ {\rm S}_{06},\allowbreak\ {\rm S}_{07},\allowbreak\ {\rm S}_{10},\allowbreak\ {\rm S}_{11},\allowbreak\ {\rm S}_{12},\allowbreak\ {\rm S}_{13},\allowbreak\ {\rm S}_{14},\allowbreak\ {\rm S}_{15},\allowbreak\ {\rm S}_{16}\big\}.$
\end{theoremB1}

\begin{proof}
By Lemma~\ref{lem_der}, the dimensions of the orbits of the superalgebras and families defining the components are as follows: 

\begin{center}
${\rm dim}\,\overline{\bigcup_{\alpha}{\mathcal O}({\rm S}_{01}^{\alpha})}=\allowbreak {\rm dim}\,{\mathcal O}({\rm S}_{03})=\allowbreak {\rm dim}\,{\mathcal O}({\rm S}_{06})=\allowbreak {\rm dim}\,{\mathcal O}({\rm S}_{07})=\allowbreak {\rm dim}\,\overline{\bigcup_{\alpha}{\mathcal O}({\rm S}_{09}^{\alpha})}=4$, \\ 
${\rm dim}\,{\mathcal O}({\rm S}_{10})=\allowbreak {\rm dim}\,{\mathcal O}({\rm S}_{11})=\allowbreak {\rm dim}\,{\mathcal O}({\rm S}_{12})=\allowbreak {\rm dim}\,{\mathcal O}({\rm S}_{13})=\allowbreak {\rm dim}\,{\mathcal O}({\rm S}_{14})=3$, \\ ${\rm dim}\,{\mathcal O}({\rm R}_{19})=\allowbreak {\rm dim}\,{\mathcal O}({\rm R}_{24})=\allowbreak {\rm dim}\,{\mathcal O}({\rm S}_{15})=\allowbreak {\rm dim}\,{\mathcal O}({\rm S}_{16})=1$.
\end{center}

\noindent{\bf Degenerations.} Every superalgebra of type $(1,2)$ that is not listed in the theorem is obtained from one of the listed superalgebras or families by one of the degenerations of Table~\ref{tab:deg12} or by a chain of them (and the superalgebra with zero multiplication is a degeneration of every superalgebra); hence $\mathfrak{SA}_{(1,2)}=\bigcup_i\mathcal{C}_i$.
\begingroup
\small
\setlength{\tabcolsep}{4pt}
\begin{longtable}{|r@{\;}c@{\;}l|l|l|l|}
\caption{Degenerations showing that every superalgebra of type $(1,2)$ lies in one of the irreducible components of $\mathfrak{SA}_{(1,2)}$.}\label{tab:deg12}\\
\hline
\multicolumn{3}{|c|}{Degeneration} & \multicolumn{3}{c|}{Parametrized basis} \\
\hline
\endfirsthead
\hline
\multicolumn{3}{|c|}{Degeneration} & \multicolumn{3}{c|}{Parametrized basis} \\
\hline
\endhead
${\rm R}_{17}$ & $\to$ & ${\rm R}_{06}^{0}$ & ${\rm E}^t_1=te_{1}$ & ${\rm F}^t_1=f_{1} + f_{2}$ & ${\rm F}^t_2=tf_{1}$ \\
\hline
${\rm S}_{01}^{\alpha}$ & $\to$ & ${\rm S}_{02}$ & ${\rm E}^t_1=2e_{1}$ & ${\rm F}^t_1=tf_{1}$ & ${\rm F}^t_2=f_{1} + f_{2}$ \\
\hline
${\rm S}_{01}^{t^{-1}}$ & $\to$ & ${\rm S}_{05}$ & ${\rm E}^t_1=t^{-1}e_{1}$ & ${\rm F}^t_1=f_{1}$ & ${\rm F}^t_2=\frac{1}{2}t^{-1}f_{1} + f_{2}$ \\
\hline
${\rm S}_{03}$ & $\to$ & ${\rm R}_{17}$ & ${\rm E}^t_1=e_{1}$ & ${\rm F}^t_1=tf_{2}$ & ${\rm F}^t_2=f_{1}$ \\
\hline
${\rm S}_{03}$ & $\to$ & ${\rm S}_{04}$ & ${\rm E}^t_1=t^{2}e_{1}$ & ${\rm F}^t_1=t^{3}f_{2}$ & ${\rm F}^t_2=f_{1} + tf_{2}$ \\
\hline
${\rm S}_{05}$ & $\to$ & ${\rm R}_{04}^{0}$ & ${\rm E}^t_1=e_{1}$ & ${\rm F}^t_1=t^{-1}f_{1}$ & ${\rm F}^t_2=tf_{2}$ \\
\hline
${\rm S}_{09}^{\alpha}$ & $\to$ & ${\rm S}_{08}^{\alpha}$ & ${\rm E}^t_1=e_{1}$ & ${\rm F}^t_1=tf_{1}$ & ${\rm F}^t_2=tf_{2}$ \\
\hline
\end{longtable}
\endgroup
\medskip

\noindent{\bf Non-degenerations.} By Section~\ref{sec:method}, it remains to prove $A\not\to B$ for the superalgebras and families $A$, $B$ defining different components with ${\rm dim}\,A>{\rm dim}\,B$. All these  non-degenerations are given in Table~\ref{tab:nondeg12}.
\begingroup
\footnotesize
\renewcommand{\arraystretch}{1.3}\setlength{\tabcolsep}{2pt}
\begin{longtable}{|>{\raggedright\arraybackslash}p{4.7cm}|>{\raggedright\arraybackslash}p{\dimexpr\textwidth-4.7cm-4\tabcolsep-3\arrayrulewidth\relax}|}
\caption{Non-degenerations between the superalgebras and families defining the irreducible components of $\mathfrak{SA}_{(1,2)}$.}\label{tab:nondeg12}\\
\hline
Non-degeneration & Reason\\
\hline
\endfirsthead
\hline
Non-degeneration & Reason\\
\hline
\endhead
${\rm S}_{03}\not\to\left\{\begin{array}{@{}l@{}}{\rm R}_{19},\ {\rm R}_{24},\ {\rm S}_{10},\\ {\rm S}_{11},\ {\rm S}_{12},\ {\rm S}_{13},\\ {\rm S}_{14},\ {\rm S}_{15},\ {\rm S}_{16}\end{array}\right\}$ & $\mathcal{R}=\left\{\begin{array}{@{}l@{}}A\mid A=\langle x_1,x_2,x_3\rangle,\ A_{1}A_{3}=0,\ A_{3}A_{1}=0,\ c_{11}^{1}=c_{21}^{2},\\ c_{12}^{2}=c_{21}^{2},\ c_{12}^{3}=c_{21}^{3}\end{array}\right\}$\newline \mbox{${\rm S}_{03}\in\mathcal{R}$ (take $x_{1}=e_1$, $x_{2}=f_2$, $x_{3}=f_1$),} \mbox{but $B\cap\mathcal{R}=\varnothing$ for all listed $B$.} \\
\hline
${\rm S}_{06}\not\to\left\{\begin{array}{@{}l@{}}{\rm R}_{19},\ {\rm R}_{24},\ {\rm S}_{10},\\ {\rm S}_{11},\ {\rm S}_{12},\ {\rm S}_{13},\\ {\rm S}_{14},\ {\rm S}_{15},\ {\rm S}_{16}\end{array}\right\}$ & $\mathcal{R}=\left\{\begin{array}{@{}l@{}}A\mid A=\langle x_1,x_2,x_3\rangle,\ A_{3}A_{3}\subseteq A_{2},\ c_{11}^{1}=0,\\ 2c_{12}^{2}=(1-\sqrt{3}i)c_{31}^{3},\ 2c_{12}^{3}=(-1+\sqrt{3}i)c_{21}^{3},\\ 2c_{13}^{2}=(-1+\sqrt{3}i)c_{31}^{2},\ 2c_{13}^{3}=(-1+\sqrt{3}i)c_{31}^{3},\\ c_{21}^{2}=-c_{31}^{3},\ 2c_{23}^{1}=(1+\sqrt{3}i)c_{32}^{1}\end{array}\right\}$\newline \mbox{${\rm S}_{06}\in\mathcal{R}$ (take $x_{1}=e_1$, $x_{2}=f_1$, $x_{3}=f_2$),} \mbox{but $B\cap\mathcal{R}=\varnothing$ for all listed $B$.} \\
\hline
${\rm S}_{07}\not\to\left\{\begin{array}{@{}l@{}}{\rm R}_{19},\ {\rm R}_{24},\ {\rm S}_{10},\\ {\rm S}_{11},\ {\rm S}_{12},\ {\rm S}_{13},\\ {\rm S}_{14},\ {\rm S}_{15},\ {\rm S}_{16}\end{array}\right\}$ & $\mathcal{R}=\left\{\begin{array}{@{}l@{}}A\mid A=\langle x_1,x_2,x_3\rangle,\ A_{3}A_{3}\subseteq A_{2},\ c_{11}^{1}=0,\\ 2c_{12}^{2}=(1+\sqrt{3}i)c_{31}^{3},\ 2c_{12}^{3}=(-1-\sqrt{3}i)c_{21}^{3},\\ 2c_{13}^{2}=(-1-\sqrt{3}i)c_{31}^{2},\ 2c_{13}^{3}=(-1-\sqrt{3}i)c_{31}^{3},\\ c_{21}^{2}=-c_{31}^{3},\ 2c_{23}^{1}=(1-\sqrt{3}i)c_{32}^{1}\end{array}\right\}$\newline \mbox{${\rm S}_{07}\in\mathcal{R}$ (take $x_{1}=e_1$, $x_{2}=f_1$, $x_{3}=f_2$),} \mbox{but $B\cap\mathcal{R}=\varnothing$ for all listed $B$.} \\
\hline
$\{{\rm S}_{01}^{\alpha}\}_{\alpha\in\mathbb{C}}\not\to\left\{\begin{array}{@{}l@{}}{\rm R}_{19},\ {\rm R}_{24},\ {\rm S}_{10},\\ {\rm S}_{11},\ {\rm S}_{12},\ {\rm S}_{13},\\ {\rm S}_{14},\ {\rm S}_{15},\ {\rm S}_{16}\end{array}\right\}$ & $\mathcal{R}=\left\{\begin{array}{@{}l@{}}A\mid A=\langle x_1,x_2,x_3\rangle,\ c_{11}^{1}=c_{12}^{3}=c_{13}^{2}=c_{21}^{3}=c_{31}^{2}=0,\\ c_{12}^{2}=c_{13}^{3},\ c_{21}^{2}=c_{31}^{3}\end{array}\right\}$\newline \mbox{${\rm S}_{01}^{\alpha}\in\mathcal{R}$ for all $\alpha$ (take $x_{1}=e_1$, $x_{2}=f_1$, $x_{3}=f_2$),} \mbox{but $B\cap\mathcal{R}=\varnothing$ for all listed $B$.} \\
\hline
$\{{\rm S}_{09}^{\alpha}\}_{\alpha\in\mathbb{C}}\not\to\left\{\begin{array}{@{}l@{}}{\rm R}_{19},\ {\rm R}_{24},\ {\rm S}_{10},\\ {\rm S}_{11},\ {\rm S}_{12},\ {\rm S}_{13},\\ {\rm S}_{14},\ {\rm S}_{15},\ {\rm S}_{16}\end{array}\right\}$ & $\mathcal{R}=\left\{\begin{array}{@{}l@{}}A\mid A=\langle x_1,x_2,x_3\rangle,\ A_{1}A_{3}\subseteq A_{2},\ A_{3}A_{1}\subseteq A_{2},\ c_{11}^{1}=0\end{array}\right\}$\newline \mbox{${\rm S}_{09}^{\alpha}\in\mathcal{R}$ for all $\alpha$ (take $x_{1}=e_1$, $x_{2}=f_2$, $x_{3}=f_1$),} \mbox{but $B\cap\mathcal{R}=\varnothing$ for all listed $B$.} \\
\hline
${\rm S}_{10}\not\to\left\{\begin{array}{@{}l@{}}{\rm R}_{19},\ {\rm R}_{24},\ {\rm S}_{15},\  {\rm S}_{16}\end{array}\right\}$ & $\mathcal{R}=\left\{\begin{array}{@{}l@{}}A\mid A=\langle x_1,x_2,x_3\rangle,\ A_{1}A_{2}\subseteq A_{2},\ c_{11}^{1}=c_{21}^{2}+c_{31}^{3},\\ c_{12}^{2}=c_{31}^{3},\ c_{12}^{3}=-c_{21}^{3},\ c_{13}^{2}=-c_{31}^{2},\ c_{13}^{3}=c_{21}^{2}\end{array}\right\}$\newline \mbox{${\rm S}_{10}\in\mathcal{R}$ (take $x_{1}=e_1$, $x_{2}=f_2$, $x_{3}=f_1+f_2$),} \mbox{but $B\cap\mathcal{R}=\varnothing$ for all listed $B$.} \\
\hline
${\rm S}_{11}\not\to\left\{\begin{array}{@{}l@{}}{\rm R}_{19},\ {\rm R}_{24},\ {\rm S}_{15},\  {\rm S}_{16}\end{array}\right\}$ & $\mathcal{R}=\left\{\begin{array}{@{}l@{}}A\mid A=\langle x_1,x_2,x_3\rangle,\ A_{1}A_{2}\subseteq A_{2},\ A_{1}A_{3}\subseteq A_{3},\ c_{12}^{3}=0,\\ c_{11}^{1}=c_{21}^{2}+c_{31}^{3},\ c_{12}^{2}=c_{21}^{2}+c_{31}^{3},\ c_{13}^{3}=c_{21}^{2}+c_{31}^{3}\end{array}\right\}$\newline \mbox{${\rm S}_{11}\in\mathcal{R}$ (take $x_{1}=e_1$, $x_{2}=f_2$, $x_{3}=f_1+f_2$),} \mbox{but $B\cap\mathcal{R}=\varnothing$ for all listed $B$.} \\
\hline
${\rm S}_{12}\not\to\left\{\begin{array}{@{}l@{}}{\rm R}_{19},\ {\rm R}_{24},\ {\rm S}_{15},\  {\rm S}_{16}\end{array}\right\}$ & $\mathcal{R}=\left\{\begin{array}{@{}l@{}}A\mid A=\langle x_1,x_2,x_3\rangle,\ A_{1}A_{2}\subseteq A_{2},\ A_{3}A_{1}\subseteq A_{3},\ c_{21}^{3}=0,\\ c_{11}^{1}=c_{31}^{3},\ c_{12}^{2}=-c_{13}^{3}+c_{31}^{3},\ c_{21}^{2}=c_{31}^{3}\end{array}\right\}$\newline \mbox{${\rm S}_{12}\in\mathcal{R}$ (take $x_{1}=e_1$, $x_{2}=f_2$, $x_{3}=f_1+f_2$),} \mbox{but $B\cap\mathcal{R}=\varnothing$ for all listed $B$.} \\
\hline
${\rm S}_{13}\not\to\left\{\begin{array}{@{}l@{}}{\rm R}_{19},\ {\rm R}_{24},\ {\rm S}_{15},\ {\rm S}_{16}\end{array}\right\}$ & $\mathcal{R}=\left\{\begin{array}{@{}l@{}}A\mid A=\langle x_1,x_2,x_3\rangle,\ A_{2}A_{1}=0,\ c_{11}^{1}=c_{12}^{2}+c_{13}^{3}\end{array}\right\}$\newline \mbox{${\rm S}_{13}\in\mathcal{R}$ (take $x_{1}=e_1$, $x_{2}=f_2$, $x_{3}=f_1+f_2$),} \mbox{but $B\cap\mathcal{R}=\varnothing$ for all listed $B$.} \\
\hline
${\rm S}_{14}\not\to\left\{\begin{array}{@{}l@{}}{\rm R}_{19},\ {\rm R}_{24},\ {\rm S}_{15},\  {\rm S}_{16}\end{array}\right\}$ & $\mathcal{R}=\left\{\begin{array}{@{}l@{}}A\mid A=\langle x_1,x_2,x_3\rangle,\ A_{1}A_{2}=0,\ c_{11}^{1}=c_{21}^{2}+c_{31}^{3}\end{array}\right\}$\newline \mbox{${\rm S}_{14}\in\mathcal{R}$ (take $x_{1}=e_1$, $x_{2}=f_2$, $x_{3}=f_1+f_2$),} \mbox{but $B\cap\mathcal{R}=\varnothing$ for all listed $B$.} \\
\hline
\end{longtable}
\endgroup
\end{proof}

\subsection{The geometric classification of semi-alternative superalgebras of type \texorpdfstring{$(2,1)$}{(2,1)}}\label{sec:deg21}
The semi-alternative superalgebras of type $(2,1)$ are listed in Theorem~A2 (and Proposition~\ref{AC(2,1)}).

\begin{theoremB2}\label{theoremB2}
The variety $\mathfrak{SA}_{(2,1)}$ has dimension $5$ and $13$ irreducible components:
$$
\begin{array}{@{}l@{\quad}l@{\quad}l@{\quad}l@{\quad}l@{}}
\mathcal{C}_{01}=\overline{{\mathcal O}({\bf R}_{01})}; & \mathcal{C}_{04}=\overline{{\mathcal O}({\bf S}_{07})}; & \mathcal{C}_{07}=\overline{{\mathcal O}({\bf S}_{14})}; & \mathcal{C}_{10}=\overline{{\mathcal O}({\bf S}_{17})}; & \mathcal{C}_{13}=\overline{{\mathcal O}({\bf S}_{27})}.\\[4pt]
\mathcal{C}_{02}=\overline{{\mathcal O}({\bf S}_{03})}; & \mathcal{C}_{05}=\overline{{\mathcal O}({\bf S}_{08})}; & \mathcal{C}_{08}=\overline{{\mathcal O}({\bf S}_{15})}; & \mathcal{C}_{11}=\overline{{\mathcal O}({\bf S}_{22})};\\[4pt]
\mathcal{C}_{03}=\overline{{\mathcal O}({\bf S}_{04})}; & \mathcal{C}_{06}=\overline{{\mathcal O}({\bf S}_{13})}; & \mathcal{C}_{09}=\overline{{\mathcal O}({\bf S}_{16})}; & \mathcal{C}_{12}=\overline{{\mathcal O}({\bf S}_{23})};
\end{array}
$$
${\rm Rig}(\mathfrak{SA}_{(2,1)})=\big\{{\bf R}_{01},\allowbreak\ {\bf S}_{03},\allowbreak\ {\bf S}_{04},\allowbreak\ {\bf S}_{07},\allowbreak\ {\bf S}_{08},\allowbreak\ {\bf S}_{13},\allowbreak\ {\bf S}_{14},\allowbreak\ {\bf S}_{15},\allowbreak\ {\bf S}_{16},\allowbreak\ {\bf S}_{17},\allowbreak\ {\bf S}_{22},\allowbreak\ {\bf S}_{23},\allowbreak\ {\bf S}_{27}\big\}.$
\end{theoremB2}

\begin{proof}
By Lemma~\ref{lem_der}, the dimensions of the orbits of the superalgebras and families defining the components are as follows: 
\begin{center}
    ${\rm dim}\,{\mathcal O}({\bf S}_{27})=5$, \\ 
    ${\rm dim}\,{\mathcal O}({\bf R}_{01})=\allowbreak {\rm dim}\,{\mathcal O}({\bf S}_{03})=\allowbreak {\rm dim}\,{\mathcal O}({\bf S}_{04})=\allowbreak {\rm dim}\,{\mathcal O}({\bf S}_{17})=\allowbreak {\rm dim}\,{\mathcal O}({\bf S}_{22})=\allowbreak {\rm dim}\,{\mathcal O}({\bf S}_{23})=4$, \\ ${\rm dim}\,{\mathcal O}({\bf S}_{07})=\allowbreak {\rm dim}\,{\mathcal O}({\bf S}_{08})=\allowbreak {\rm dim}\,{\mathcal O}({\bf S}_{13})=\allowbreak {\rm dim}\,{\mathcal O}({\bf S}_{14})=\allowbreak {\rm dim}\,{\mathcal O}({\bf S}_{15})=\allowbreak {\rm dim}\,{\mathcal O}({\bf S}_{16})=3$.

\end{center}

\noindent{\bf Degenerations.} Every superalgebra of type $(2,1)$ that is not listed in the theorem is obtained from one of the listed superalgebras or families by one of the degenerations of Table~\ref{tab:deg21} or by a chain of them (and the superalgebra with zero multiplication is a degeneration of every superalgebra); hence $\mathfrak{SA}_{(2,1)}=\bigcup_i\mathcal{C}_i$.
\begingroup
\small
\setlength{\tabcolsep}{4pt}
\begin{longtable}{|r@{\;}c@{\;}l|l|l|l|}
\caption{Degenerations showing that every superalgebra of type $(2,1)$ lies in one of the irreducible components of $\mathfrak{SA}_{(2,1)}$.}\label{tab:deg21}\\
\hline
\multicolumn{3}{|c|}{Degeneration} & \multicolumn{3}{c|}{Parametrized basis} \\
\hline
\endfirsthead
\hline
\multicolumn{3}{|c|}{Degeneration} & \multicolumn{3}{c|}{Parametrized basis} \\
\hline
\endhead
${\bf R}_{01}$ & $\to$ & ${\bf R}_{09}$ & ${\rm E}^t_1=e_{2}$ & ${\rm E}^t_2=-te_{1} + te_{2}$ & ${\rm F}^t_1=t^{-2}f_{1}$ \\
\hline
${\bf R}_{01}$ & $\to$ & ${\bf R}_{17}$ & ${\rm E}^t_1=e_{1} + e_{2}$ & ${\rm E}^t_2=te_{2}$ & ${\rm F}^t_1=t^{-2}f_{1}$ \\
\hline
${\bf R}_{02}$ & $\to$ & ${\bf R}_{15}$ & ${\rm E}^t_1=e_{1}$ & ${\rm E}^t_2=te_{1} + te_{2}$ & ${\rm F}^t_1=t^{-2}f_{1}$ \\
\hline
${\bf S}_{01}$ & $\to$ & ${\bf S}_{24}$ & ${\rm E}^t_1=t^{-2}e_{1}$ & ${\rm E}^t_2=te_{1} + 2te_{2}$ & ${\rm F}^t_1=t^{-1}f_{1}$ \\
\hline
${\bf S}_{01}$ & $\to$ & ${\bf S}_{31}$ & ${\rm E}^t_1=t^{-2}e_{1}$ & ${\rm E}^t_2=\frac{1}{2}e_{1} + e_{2}$ & ${\rm F}^t_1=f_{1}$ \\
\hline
${\bf S}_{02}$ & $\to$ & ${\bf S}_{30}$ & ${\rm E}^t_1=t^{-2}e_{1}$ & ${\rm E}^t_2=-\frac{1}{2}e_{1} + e_{2}$ & ${\rm F}^t_1=f_{1}$ \\
\hline
${\bf S}_{03}$ & $\to$ & ${\bf S}_{01}$ & ${\rm E}^t_1=4t^{-2}e_{1}$ & ${\rm E}^t_2=2e_{1} + e_{2}$ & ${\rm F}^t_1=2t^{-1}f_{1}$ \\
\hline
${\bf S}_{03}$ & $\to$ & ${\bf S}_{29}$ & ${\rm E}^t_1=2t^{2}e_{2}$ & ${\rm E}^t_2=4te_{1} - 2te_{2}$ & ${\rm F}^t_1=2tf_{1}$ \\
\hline
${\bf S}_{03}$ & $\to$ & ${\bf S}_{32}$ & ${\rm E}^t_1=e_{1}$ & ${\rm E}^t_2=2e_{1} + e_{2}$ & ${\rm F}^t_1=2tf_{1}$ \\
\hline
${\bf S}_{04}$ & $\to$ & ${\bf S}_{02}$ & ${\rm E}^t_1=4t^{-2}e_{1}$ & ${\rm E}^t_2=-2e_{1} + e_{2}$ & ${\rm F}^t_1=2t^{-1}f_{1}$ \\
\hline
${\bf S}_{04}$ & $\to$ & ${\bf S}_{33}$ & ${\rm E}^t_1=e_{1}$ & ${\rm E}^t_2=-2e_{1} + e_{2}$ & ${\rm F}^t_1=2tf_{1}$ \\
\hline
${\bf S}_{07}$ & $\to$ & ${\bf R}_{37}$ & ${\rm E}^t_1=2te_{2}$ & ${\rm E}^t_2=4t^{2}e_{1} + 2t^{2}e_{2}$ & ${\rm F}^t_1=-t^{-2}f_{1}$ \\
\hline
${\bf S}_{07}$ & $\to$ & ${\bf S}_{05}$ & ${\rm E}^t_1=4t^{-2}e_{1}$ & ${\rm E}^t_2=2e_{1} + e_{2}$ & ${\rm F}^t_1=t^{-2}f_{1}$ \\
\hline
${\bf S}_{08}$ & $\to$ & ${\bf S}_{06}$ & ${\rm E}^t_1=4t^{-2}e_{1}$ & ${\rm E}^t_2=-2e_{1} + e_{2}$ & ${\rm F}^t_1=t^{-2}f_{1}$ \\
\hline
${\bf S}_{13}$ & $\to$ & ${\bf S}_{09}$ & ${\rm E}^t_1=4t^{-2}e_{1}$ & ${\rm E}^t_2=2e_{1} + e_{2}$ & ${\rm F}^t_1=t^{-2}f_{1}$ \\
\hline
${\bf S}_{14}$ & $\to$ & ${\bf S}_{10}$ & ${\rm E}^t_1=4t^{-2}e_{1}$ & ${\rm E}^t_2=2e_{1} + e_{2}$ & ${\rm F}^t_1=t^{-2}f_{1}$ \\
\hline
${\bf S}_{15}$ & $\to$ & ${\bf S}_{11}$ & ${\rm E}^t_1=4t^{-2}e_{1}$ & ${\rm E}^t_2=-2e_{1} + e_{2}$ & ${\rm F}^t_1=t^{-2}f_{1}$ \\
\hline
${\bf S}_{16}$ & $\to$ & ${\bf S}_{12}$ & ${\rm E}^t_1=4t^{-2}e_{1}$ & ${\rm E}^t_2=-2e_{1} + e_{2}$ & ${\rm F}^t_1=t^{-2}f_{1}$ \\
\hline
${\bf S}_{17}$ & $\to$ & ${\bf R}_{21}$ & ${\rm E}^t_1=e_{2}$ & ${\rm E}^t_2=-te_{1} + te_{2}$ & ${\rm F}^t_1=t^{-2}f_{1}$ \\
\hline
${\bf S}_{17}$ & $\to$ & ${\bf S}_{20}$ & ${\rm E}^t_1=\frac{1}{2}e_{1} + \frac{1}{2}e_{2}$ & ${\rm E}^t_2=te_{2}$ & ${\rm F}^t_1=t^{-2}f_{1}$ \\
\hline
${\bf S}_{17}$ & $\to$ & ${\bf S}_{21}$ & ${\rm E}^t_1=-\frac{1}{2}e_{1} + \frac{1}{2}e_{2}$ & ${\rm E}^t_2=te_{2}$ & ${\rm F}^t_1=t^{-2}f_{1}$ \\
\hline
${\bf S}_{22}$ & $\to$ & ${\bf S}_{18}$ & ${\rm E}^t_1=e_{1} + e_{2}$ & ${\rm E}^t_2=te_{2}$ & ${\rm F}^t_1=t^{-2}f_{1}$ \\
\hline
${\bf S}_{23}$ & $\to$ & ${\bf S}_{19}$ & ${\rm E}^t_1=e_{1} + e_{2}$ & ${\rm E}^t_2=te_{2}$ & ${\rm F}^t_1=t^{-2}f_{1}$ \\
\hline
${\bf S}_{27}$ & $\to$ & ${\bf R}_{02}$ & ${\rm E}^t_1=e_{1}$ & ${\rm E}^t_2=e_{2}$ & ${\rm F}^t_1=tf_{1}$ \\
\hline
${\bf S}_{27}$ & $\to$ & ${\bf S}_{25}$ & ${\rm E}^t_1=2t^{2}e_{1} + t^{2}e_{2}$ & ${\rm E}^t_2=e_{1} + e_{2}$ & ${\rm F}^t_1=tf_{1}$ \\
\hline
${\bf S}_{27}$ & $\to$ & ${\bf S}_{26}$ & ${\rm E}^t_1=e_{1}$ & ${\rm E}^t_2=te_{1} - te_{2}$ & ${\rm F}^t_1=-f_{1}$ \\
\hline
${\bf S}_{27}$ & $\to$ & ${\bf S}_{28}$ & ${\rm E}^t_1=t^{2}e_{1} + t^{2}e_{2}$ & ${\rm E}^t_2=e_{2}$ & ${\rm F}^t_1=tf_{1}$ \\
\hline
\end{longtable}
\endgroup
\medskip

\noindent{\bf Non-degenerations.} By Section~\ref{sec:method}, it remains to prove $A\not\to B$ for the superalgebras and families $A$, $B$ defining different components with ${\rm dim}\,A>{\rm dim}\,B$. All these non-degenerations are given in Table~\ref{tab:nondeg21}.
\begingroup
\footnotesize
\renewcommand{\arraystretch}{1.3}\setlength{\tabcolsep}{2pt}
\begin{longtable}{|>{\raggedright\arraybackslash}p{4.7cm}|>{\raggedright\arraybackslash}p{\dimexpr\textwidth-4.7cm-4\tabcolsep-3\arrayrulewidth\relax}|}
\caption{Non-degenerations between the superalgebras and families defining the irreducible components of $\mathfrak{SA}_{(2,1)}$.}\label{tab:nondeg21}\\
\hline
Non-degeneration & Reason\\
\hline
\endfirsthead
\hline
Non-degeneration & Reason\\
\hline
\endhead
${\bf S}_{27}\not\to\left\{\begin{array}{@{}l@{}}
{\bf R}_{01},\ {\bf S}_{03},\ {\bf S}_{04},\ {\bf S}_{07},\\ 
{\bf S}_{08},\ {\bf S}_{13},\ {\bf S}_{14},\ {\bf S}_{15},\\ 
{\bf S}_{16},\ {\bf S}_{17},\ {\bf S}_{22},\ {\bf S}_{23}\end{array}\right\}$ & $\mathcal{R}=\left\{\begin{array}{@{}l@{}}A\mid A=\langle x_1,x_2,x_3\rangle,\ c_{12}^{1}=c_{21}^{1}=c_{22}^{1}=c_{23}^{3}=c_{32}^{3}=0,\\ c_{11}^{1}=c_{31}^{3},\ c_{12}^{2}=c_{21}^{2},\ c_{13}^{3}=c_{31}^{3}\end{array}\right\}$\newline \mbox{${\bf S}_{27}\in\mathcal{R}$ (take $x_{1}=e_1$, $x_{2}=e_2$, $x_{3}=f_1$),} \mbox{but $B\cap\mathcal{R}=\varnothing$ for all listed $B$.} \\
\hline
${\bf R}_{01}\not\to\left\{\begin{array}{@{}l@{}}{\bf S}_{07},\ {\bf S}_{08},\ {\bf S}_{13},\\ {\bf S}_{14},\ {\bf S}_{15},\ {\bf S}_{16}\end{array}\right\}$ & $\mathcal{R}=\left\{\begin{array}{@{}l@{}}A\mid A=\langle x_1,x_2,x_3\rangle,\ A_{1}A_{3}=0,\ A_{3}A_{1}=0,\ c_{12}^{1}=c_{21}^{1},\\ c_{12}^{2}=c_{21}^{2}\end{array}\right\}$\newline \mbox{${\bf R}_{01}\in\mathcal{R}$ (take $x_{1}=e_2$, $x_{2}=e_1+2e_2$, $x_{3}=f_1$),} \mbox{but $B\cap\mathcal{R}=\varnothing$ for all listed $B$.} \\
\hline
${\bf S}_{03}\not\to\left\{\begin{array}{@{}l@{}}{\bf S}_{07},\ {\bf S}_{08},\ {\bf S}_{13},\\ {\bf S}_{14},\ {\bf S}_{15},\ {\bf S}_{16}\end{array}\right\}$ & $\mathcal{R}=\left\{\begin{array}{@{}l@{}}A\mid A=\langle x_1,x_2,x_3\rangle,\ c_{31}^{3}=c_{32}^{3}=0,\ c_{11}^{1}=2c_{13}^{3}-c_{21}^{2},\\ c_{12}^{1}=-c_{22}^{2}+2c_{23}^{3},\ c_{12}^{2}=-c_{13}^{3}+c_{21}^{2},\\ c_{21}^{1}=-c_{22}^{2}+c_{23}^{3}\end{array}\right\}$\newline \mbox{${\bf S}_{03}\in\mathcal{R}$ (take $x_{1}=e_1$, $x_{2}=e_2$, $x_{3}=f_1$),} \mbox{but $B\cap\mathcal{R}=\varnothing$ for all listed $B$.} \\
\hline
${\bf S}_{04}\not\to\left\{\begin{array}{@{}l@{}}{\bf S}_{07},\ {\bf S}_{08},\ {\bf S}_{13},\\ {\bf S}_{14},\ {\bf S}_{15},\ {\bf S}_{16}\end{array}\right\}$ & $\mathcal{R}=\left\{\begin{array}{@{}l@{}}A\mid A=\langle x_1,x_2,x_3\rangle,\ c_{13}^{3}=c_{23}^{3}=0,\ c_{11}^{1}=-c_{21}^{2}+c_{31}^{3},\\ c_{12}^{1}=-c_{22}^{2}+c_{32}^{3},\ c_{12}^{2}=c_{21}^{2}+c_{31}^{3},\ c_{21}^{1}=-c_{22}^{2}+2c_{32}^{3}\end{array}\right\}$\newline \mbox{${\bf S}_{04}\in\mathcal{R}$ (take $x_{1}=e_1$, $x_{2}=e_2$, $x_{3}=f_1$),} \mbox{but $B\cap\mathcal{R}=\varnothing$ for all listed $B$.} \\
\hline
${\bf S}_{17}\not\to\left\{\begin{array}{@{}l@{}}{\bf S}_{07},\ {\bf S}_{08},\ {\bf S}_{13},\\ {\bf S}_{14},\ {\bf S}_{15},\ {\bf S}_{16}\end{array}\right\}$ & $\mathcal{R}=\left\{\begin{array}{@{}l@{}}A\mid A=\langle x_1,x_2,x_3\rangle,\ A_{1}A_{3}\subseteq A_{3},\ c_{11}^{1}=c_{13}^{3}-c_{21}^{2}+c_{31}^{3},\\ c_{12}^{1}=-c_{22}^{2},\ c_{12}^{2}=c_{21}^{2},\ c_{21}^{1}=-c_{22}^{2},\ c_{23}^{3}=-c_{32}^{3}\end{array}\right\}$\newline \mbox{${\bf S}_{17}\in\mathcal{R}$ (take $x_{1}=e_2$, $x_{2}=e_1$, $x_{3}=f_1$),} \mbox{but $B\cap\mathcal{R}=\varnothing$ for all listed $B$.} \\
\hline
${\bf S}_{22}\not\to\left\{\begin{array}{@{}l@{}}{\bf S}_{07},\ {\bf S}_{08},\ {\bf S}_{13},\\ {\bf S}_{14},\ {\bf S}_{15},\ {\bf S}_{16}\end{array}\right\}$ & $\mathcal{R}=\left\{\begin{array}{@{}l@{}}A\mid A=\langle x_1,x_2,x_3\rangle,\ A_{1}A_{2}\subseteq A_{2},\ A_{2}A_{1}\subseteq A_{2},\\ A_{2}A_{3}=0,\ A_{3}A_{1}=0,\ c_{11}^{1}=c_{13}^{3},\ c_{12}^{2}=c_{21}^{2}\end{array}\right\}$\newline \mbox{${\bf S}_{22}\in\mathcal{R}$ (take $x_{1}=e_1$, $x_{2}=e_2$, $x_{3}=f_1$),} \mbox{but $B\cap\mathcal{R}=\varnothing$ for all listed $B$.} \\
\hline
${\bf S}_{23}\not\to\left\{\begin{array}{@{}l@{}}{\bf S}_{07},\ {\bf S}_{08},\ {\bf S}_{13},\\ {\bf S}_{14},\ {\bf S}_{15},\ {\bf S}_{16}\end{array}\right\}$ & $\mathcal{R}=\left\{\begin{array}{@{}l@{}}A\mid A=\langle x_1,x_2,x_3\rangle,\ A_{1}A_{2}\subseteq A_{2},\ A_{1}A_{3}=0,\\ A_{2}A_{1}\subseteq A_{2},\ A_{3}A_{2}=0,\ c_{11}^{1}=c_{31}^{3},\ c_{12}^{2}=c_{21}^{2}\end{array}\right\}$\newline \mbox{${\bf S}_{23}\in\mathcal{R}$ (take $x_{1}=e_1$, $x_{2}=e_2$, $x_{3}=f_1$),} \mbox{but $B\cap\mathcal{R}=\varnothing$ for all listed $B$.} \\
\hline
\end{longtable}
\endgroup
\end{proof}

\subsection{Subvarieties}\label{sec:subvar}
By Section~\ref{sec:method}, a degeneration $A\to B$ (resp. a non-degeneration $A\not\to B$) of semi-alternative superalgebras with $A,B$ from a subvariety is also a degeneration (resp. a non-degeneration) in the subvariety. Below, together with the degenerations and non-degenerations of Tables~\ref{tab:deg12}--\ref{tab:nondeg21}, we use the degenerations of Tables~\ref{tab:subdeg12} and~\ref{tab:subdeg21} and the non-degenerations of Tables~\ref{tab:subnondeg12} and~\ref{tab:subnondeg21}.
\begingroup
\small
\setlength{\tabcolsep}{4pt}
\begin{longtable}{|r@{\;}c@{\;}l|l|l|l|}
\caption{Additional degenerations used for the subvarieties of $\mathfrak{SA}_{(1,2)}$.}\label{tab:subdeg12}\\
\hline
\multicolumn{3}{|c|}{Degeneration} & \multicolumn{3}{c|}{Parametrized basis} \\
\hline
\endfirsthead
\hline
\multicolumn{3}{|c|}{Degeneration} & \multicolumn{3}{c|}{Parametrized basis} \\
\hline
\endhead
${\rm S}_{08}^{t^{-1}}$ & $\to$ & ${\rm R}_{06}^{0}$ & ${\rm E}^t_1=e_{1}$ & ${\rm F}^t_1=f_{2}$ & ${\rm F}^t_2=t^{-1}f_{1}$ \\
\hline
\end{longtable}
\endgroup
\begingroup
\footnotesize
\renewcommand{\arraystretch}{1.3}\setlength{\tabcolsep}{2pt}
\begin{longtable}{|>{\raggedright\arraybackslash}p{4.7cm}|>{\raggedright\arraybackslash}p{\dimexpr\textwidth-4.7cm-4\tabcolsep-3\arrayrulewidth\relax}|}
\caption{Additional non-degenerations used for the subvarieties of $\mathfrak{SA}_{(1,2)}$.}\label{tab:subnondeg12}\\
\hline
Non-degeneration & Reason\\
\hline
\endfirsthead
\hline
Non-degeneration & Reason\\
\hline
\endhead
${\rm R}_{17}\not\to\left\{\begin{array}{@{}l@{}}{\rm R}_{19},\ {\rm R}_{24},\ {\rm S}_{15}\end{array}\right\}$ & $\mathcal{R}=\left\{\begin{array}{@{}l@{}}A\mid A=\langle x_1,x_2,x_3\rangle,\ c_{11}^{1}=c_{21}^{2}+c_{31}^{3}\end{array}\right\}$\newline \mbox{${\rm R}_{17}\in\mathcal{R}$ (take $x_{1}=e_1$, $x_{2}=f_2$, $x_{3}=f_1+f_2$),} \mbox{but $B\cap\mathcal{R}=\varnothing$ for all listed $B$.} \\
\hline
$\{{\rm S}_{01}^{\alpha}\}_{\alpha\in\mathbb{C}}\not\to\left\{\begin{array}{@{}l@{}}{\rm R}_{17},\ {\rm S}_{08}^{\beta}\ (\beta\in\mathbb{C})\end{array}\right\}$ & $\mathcal{R}=\left\{\begin{array}{@{}l@{}}A\mid A=\langle x_1,x_2,x_3\rangle,\ c_{11}^{1}=c_{12}^{3}=c_{13}^{2}=c_{21}^{3}=c_{31}^{2}=0,\\ c_{12}^{2}=c_{13}^{3},\ c_{21}^{2}=c_{31}^{3}\end{array}\right\}$\newline \mbox{${\rm S}_{01}^{\alpha}\in\mathcal{R}$ for all $\alpha$ (take $x_{1}=e_1$, $x_{2}=f_1$, $x_{3}=f_2$),} \mbox{but $B\cap\mathcal{R}=\varnothing$ for all listed $B$.} \\
\hline
${\rm S}_{03}\not\to\left\{\begin{array}{@{}l@{}}{\rm S}_{08}^{\beta}\ (\beta\in\mathbb{C})\end{array}\right\}$ & $\mathcal{R}=\left\{\begin{array}{@{}l@{}}A\mid A=\langle x_1,x_2,x_3\rangle,\ A_{1}A_{3}=0,\ A_{3}A_{1}=0,\ c_{11}^{1}=c_{21}^{2},\\ c_{12}^{2}=c_{21}^{2},\ c_{12}^{3}=c_{21}^{3}\end{array}\right\}$\newline \mbox{${\rm S}_{03}\in\mathcal{R}$ (take $x_{1}=e_1$, $x_{2}=f_2$, $x_{3}=f_1$),} \mbox{but ${\rm S}_{08}^{\beta}\cap\mathcal{R}=\varnothing$.} \\
\hline
$\{{\rm S}_{08}^{\alpha}\}_{\alpha\in\mathbb{C}}\not\to\left\{\begin{array}{@{}l@{}}{\rm R}_{19},\ {\rm R}_{24},\\ {\rm S}_{15},\   {\rm S}_{16}\end{array}\right\}$ & $\mathcal{R}=\left\{\begin{array}{@{}l@{}}A\mid A=\langle x_1,x_2,x_3\rangle,\ c_{11}^{1}=0\end{array}\right\}$\newline \mbox{${\rm S}_{08}^{\alpha}\in\mathcal{R}$ for all $\alpha$ (take $x_{1}=e_1$, $x_{2}=f_1$, $x_{3}=f_2$),} \mbox{but $B\cap\mathcal{R}=\varnothing$ for all listed $B$.} \\
\hline
\end{longtable}
\endgroup
\begingroup
\small
\setlength{\tabcolsep}{4pt}
\begin{longtable}{|r@{\;}c@{\;}l|l|l|l|}
\caption{Additional degenerations used for the subvarieties of $\mathfrak{SA}_{(2,1)}$.}\label{tab:subdeg21}\\
\hline
\multicolumn{3}{|c|}{Degeneration} & \multicolumn{3}{c|}{Parametrized basis} \\
\hline
\endfirsthead
\hline
\multicolumn{3}{|c|}{Degeneration} & \multicolumn{3}{c|}{Parametrized basis} \\
\hline
\endhead
${\bf R}_{02}$ & $\to$ & ${\bf R}_{21}$ & ${\rm E}^t_1=e_{1} + e_{2}$ & ${\rm E}^t_2=te_{2}$ & ${\rm F}^t_1=t^{-2}f_{1}$ \\
\hline
${\bf R}_{09}$ & $\to$ & ${\bf R}_{37}$ & ${\rm E}^t_1=te_{1} + te_{2}$ & ${\rm E}^t_2=2t^{2}e_{1} + t^{2}e_{2}$ & ${\rm F}^t_1=-t^{-2}f_{1}$ \\
\hline
${\bf S}_{22}$ & $\to$ & ${\bf S}_{20}$ & ${\rm E}^t_1=e_{1}$ & ${\rm E}^t_2=2te_{1} + te_{2}$ & ${\rm F}^t_1=t^{-2}f_{1}$ \\
\hline
${\bf S}_{25}$ & $\to$ & ${\bf S}_{29}$ & ${\rm E}^t_1=2t^{2}e_{1} + t^{2}e_{2}$ & ${\rm E}^t_2=te_{1} + te_{2}$ & ${\rm F}^t_1=tf_{1}$ \\
\hline
${\bf S}_{28}$ & $\to$ & ${\bf S}_{29}$ & ${\rm E}^t_1=t^{2}e_{2}$ & ${\rm E}^t_2=-te_{1} + te_{2}$ & ${\rm F}^t_1=tf_{1}$ \\
\hline
${\bf S}_{29}$ & $\to$ & ${\bf S}_{24}$ & ${\rm E}^t_1=t^{-2}e_{1}$ & ${\rm E}^t_2=e_{1} + e_{2}$ & ${\rm F}^t_1=t^{-1}f_{1}$ \\
\hline
${\bf S}_{32}$ & $\to$ & ${\bf S}_{31}$ & ${\rm E}^t_1=4t^{-2}e_{1}$ & ${\rm E}^t_2=2e_{1} + e_{2}$ & ${\rm F}^t_1=t^{-2}f_{1}$ \\
\hline
${\bf S}_{33}$ & $\to$ & ${\bf S}_{30}$ & ${\rm E}^t_1=4t^{-2}e_{1}$ & ${\rm E}^t_2=-2e_{1} + e_{2}$ & ${\rm F}^t_1=t^{-2}f_{1}$ \\
\hline
\end{longtable}
\endgroup
\begingroup
\footnotesize
\renewcommand{\arraystretch}{1.3}\setlength{\tabcolsep}{2pt}
\begin{longtable}{|>{\raggedright\arraybackslash}p{4.7cm}|>{\raggedright\arraybackslash}p{\dimexpr\textwidth-4.7cm-4\tabcolsep-3\arrayrulewidth\relax}|}
\caption{Additional non-degenerations used for the subvarieties of $\mathfrak{SA}_{(2,1)}$.}\label{tab:subnondeg21}\\
\hline
Non-degeneration & Reason\\
\hline
\endfirsthead
\hline
Non-degeneration & Reason\\
\hline
\endhead
${\bf R}_{01}\not\to\left\{\begin{array}{@{}l@{}}
{\bf S}_{05},\ {\bf S}_{06},\ {\bf S}_{09},\ {\bf S}_{10},\\ 
{\bf S}_{11},\ {\bf S}_{12},\ {\bf S}_{30},\ {\bf S}_{31},\\ 
{\bf S}_{32},\ {\bf S}_{33}\end{array}\right\}$ & 
$\mathcal{R}=\left\{\begin{array}{@{}l@{}}A\mid A=\langle x_1,x_2,x_3\rangle,\ A_{1}A_{3}=0,\ A_{3}A_{1}=0,\ c_{12}^{1}=c_{21}^{1},\ c_{12}^{2}=c_{21}^{2}\end{array}\right\}$\newline \mbox{${\bf R}_{01}\in\mathcal{R}$ (take $x_{1}=e_2$, $x_{2}=e_1+2e_2$, $x_{3}=f_1$),} \mbox{but $B\cap\mathcal{R}=\varnothing$ for all listed $B$.} \\
\hline
${\bf R}_{02}\not\to\left\{\begin{array}{@{}l@{}}{\bf S}_{06},\ {\bf S}_{11},\ {\bf S}_{12}\end{array}\right\}$ & $\mathcal{R}=\left\{\begin{array}{@{}l@{}}A\mid A=\langle x_1,x_2,x_3\rangle,\ c_{12}^{1}=c_{21}^{1},\ c_{12}^{2}=c_{21}^{2}\end{array}\right\}$\newline \mbox{${\bf R}_{02}\in\mathcal{R}$ (take $x_{1}=e_1$, $x_{2}=e_2$, $x_{3}=f_1$),} \mbox{but $B\cap\mathcal{R}=\varnothing$ for all listed $B$.} \\
\hline
${\bf S}_{17}\not\to\left\{\begin{array}{@{}l@{}}
{\bf S}_{05},\ {\bf S}_{06},\ {\bf S}_{09},\  {\bf S}_{10},\\ 
{\bf S}_{11},\ {\bf S}_{12},\ {\bf S}_{30},\ {\bf S}_{31},\\ 
{\bf S}_{32},\ {\bf S}_{33}\end{array}\right\}$ & $\mathcal{R}=\left\{\begin{array}{@{}l@{}}A\mid A=\langle x_1,x_2,x_3\rangle,\ A_{1}A_{3}\subseteq A_{3},\ c_{11}^{1}=c_{13}^{3}-c_{21}^{2}+c_{31}^{3},\\ c_{12}^{1}=-c_{22}^{2},\ c_{12}^{2}=c_{21}^{2},\ c_{21}^{1}=-c_{22}^{2},\ c_{23}^{3}=-c_{32}^{3}\end{array}\right\}$\newline \mbox{${\bf S}_{17}\in\mathcal{R}$ (take $x_{1}=e_2$, $x_{2}=e_1$, $x_{3}=f_1$),} \mbox{but $B\cap\mathcal{R}=\varnothing$ for all listed $B$.} \\
\hline

${\bf S}_{22}\not\to\left\{\begin{array}{@{}l@{}}
{\bf S}_{05},\ {\bf S}_{06},\ {\bf S}_{09},\  {\bf S}_{10},\\  
{\bf S}_{11},\ {\bf S}_{12},\ {\bf S}_{30},\ {\bf S}_{31},\\ 
{\bf S}_{32},\  {\bf S}_{33}\end{array}\right\}$ & $\mathcal{R}=\left\{\begin{array}{@{}l@{}}A\mid A=\langle x_1,x_2,x_3\rangle,\ A_{1}A_{2}\subseteq A_{2},\ A_{2}A_{1}\subseteq A_{2},\\ A_{2}A_{3}=0,\ A_{3}A_{1}=0,\ c_{11}^{1}=c_{13}^{3},\ c_{12}^{2}=c_{21}^{2}\end{array}\right\}$\newline \mbox{${\bf S}_{22}\in\mathcal{R}$ (take $x_{1}=e_1$, $x_{2}=e_2$, $x_{3}=f_1$),} \mbox{but $B\cap\mathcal{R}=\varnothing$ for all listed $B$.} \\

\hline
${\bf S}_{23}\not\to\left\{\begin{array}{@{}l@{}}
{\bf S}_{05},\ {\bf S}_{06},\ {\bf S}_{09},\  {\bf S}_{10},\\ 
{\bf S}_{11},\ {\bf S}_{12},\  {\bf S}_{30},\ {\bf S}_{31},\\ 
{\bf S}_{32},\ {\bf S}_{33}\end{array}\right\}$ & $\mathcal{R}=\left\{\begin{array}{@{}l@{}}A\mid A=\langle x_1,x_2,x_3\rangle,\ A_{1}A_{2}\subseteq A_{2},\ A_{1}A_{3}=0,\\ A_{2}A_{1}\subseteq A_{2},\ A_{3}A_{2}=0,\ c_{11}^{1}=c_{31}^{3},\ c_{12}^{2}=c_{21}^{2}\end{array}\right\}$\newline \mbox{${\bf S}_{23}\in\mathcal{R}$ (take $x_{1}=e_1$, $x_{2}=e_2$, $x_{3}=f_1$),} \mbox{but $B\cap\mathcal{R}=\varnothing$ for all listed $B$.} \\
\hline
${\bf S}_{27}\not\to\left\{\begin{array}{@{}l@{}}
{\bf S}_{05},\ {\bf S}_{06},\ {\bf S}_{09},\  {\bf S}_{10},\\ 
{\bf S}_{11},\ {\bf S}_{12},\  {\bf S}_{30},\ {\bf S}_{31},\\ 
{\bf S}_{32},\  {\bf S}_{33}\end{array}\right\}$ & $\mathcal{R}=\left\{\begin{array}{@{}l@{}}A\mid A=\langle x_1,x_2,x_3\rangle,\ c_{12}^{1}=c_{21}^{1}=c_{22}^{1}=c_{23}^{3}=c_{32}^{3}=0,\\ c_{11}^{1}=c_{31}^{3},\ c_{12}^{2}=c_{21}^{2},\ c_{13}^{3}=c_{31}^{3}\end{array}\right\}$\newline \mbox{${\bf S}_{27}\in\mathcal{R}$ (take $x_{1}=e_1$, $x_{2}=e_2$, $x_{3}=f_1$),} \mbox{but $B\cap\mathcal{R}=\varnothing$ for all listed $B$.} \\
\hline
${\bf S}_{28}\not\to\left\{\begin{array}{@{}l@{}}{\bf S}_{06},\ {\bf S}_{11},\ {\bf S}_{12}\end{array}\right\}$ & $\mathcal{R}=\left\{\begin{array}{@{}l@{}}A\mid A=\langle x_1,x_2,x_3\rangle,\ c_{12}^{1}=c_{21}^{1},\ c_{12}^{2}=c_{21}^{2}\end{array}\right\}$\newline \mbox{${\bf S}_{28}\in\mathcal{R}$ (take $x_{1}=e_1$, $x_{2}=e_2$, $x_{3}=f_1$),} \mbox{but $B\cap\mathcal{R}=\varnothing$ for all listed $B$.} \\
\hline
\end{longtable}
\endgroup

\subsubsection{Associative supercommutative superalgebras}
The associative supercommutative superalgebras of types $(1,2)$ and $(2,1)$ are listed in Propositions~\ref{AC(1,2)} and~\ref{AC(2,1)}. The following result was obtained in \cite{AKK}.
\begin{corollary}[see \cite{AKK}]\label{th:ac12} The variety $\mathfrak{ACom}_{(1,2)}$ has dimension $3$ and $4$ irreducible components:
$$
\begin{array}{@{}l@{\quad}l@{\quad}l@{\quad}l@{}}
\mathcal{C}_{01}=\overline{{\mathcal O}({\rm R}_{04}^{0})}; & \mathcal{C}_{02}=\overline{{\mathcal O}({\rm R}_{17})}; & \mathcal{C}_{03}=\overline{{\mathcal O}({\rm R}_{19})}; & \mathcal{C}_{04}=\overline{{\mathcal O}({\rm R}_{24})}.
\end{array}
$$
${\rm Rig}(\mathfrak{ACom}_{(1,2)})=\big\{{\rm R}_{04}^{0},\allowbreak\ {\rm R}_{17},\allowbreak\ {\rm R}_{19},\allowbreak\ {\rm R}_{24}\big\}.$
\end{corollary}

\begin{corollary}[see \cite{AKK}]\label{th:ac21} The variety $\mathfrak{ACom}_{(2,1)}$ has dimension $4$ and $2$ irreducible components:
$$
\begin{array}{@{}l@{\quad}l@{}}
\mathcal{C}_{01}=\overline{{\mathcal O}({\bf R}_{01})}; & \mathcal{C}_{02}=\overline{{\mathcal O}({\bf R}_{02})}.
\end{array}
$$
${\rm Rig}(\mathfrak{ACom}_{(2,1)})=\big\{{\bf R}_{01},\allowbreak\ {\bf R}_{02}\big\}.$
\end{corollary}

\subsubsection{\texorpdfstring{$\mathfrak{perm}$}{perm} superalgebras}
The $\mathfrak{perm}$ superalgebras of types $(1,2)$ and $(2,1)$ are listed in Corollaries~\ref{cor:perm12} and~\ref{cor:perm21}.
\begin{corollary}\label{th:perm12} The variety $\mathfrak{Perm}_{(1,2)}$ has dimension $4$ and $8$ irreducible components:
$$
\begin{array}{@{}l@{\quad}l@{\quad}l@{\quad}l@{}}
\mathcal{C}_{01}=\overline{{\mathcal O}({\rm R}_{17})}; & \mathcal{C}_{03}=\overline{{\mathcal O}({\rm R}_{24})}; & \mathcal{C}_{05}=\overline{\bigcup_{\alpha}{\mathcal O}({\rm S}_{08}^{\alpha})}; & \mathcal{C}_{07}=\overline{{\mathcal O}({\rm S}_{13})};\\[4pt]
\mathcal{C}_{02}=\overline{{\mathcal O}({\rm R}_{19})}; & \mathcal{C}_{04}=\overline{\bigcup_{\alpha}{\mathcal O}({\rm S}_{01}^{\alpha})}; & \mathcal{C}_{06}=\overline{{\mathcal O}({\rm S}_{11})}; & \mathcal{C}_{08}=\overline{{\mathcal O}({\rm S}_{15})}.
\end{array}
$$
${\rm Rig}(\mathfrak{Perm}_{(1,2)})=\big\{{\rm R}_{17},\allowbreak\ {\rm R}_{19},\allowbreak\ {\rm R}_{24},\allowbreak\ {\rm S}_{11},\allowbreak\ {\rm S}_{13},\allowbreak\ {\rm S}_{15}\big\}.$
\end{corollary}
\begin{proof} Every superalgebra of $\mathfrak{Perm}_{(1,2)}$ that is not listed above belongs to one of the components by the degenerations ${\rm S}_{01}^{t^{-1}}\to {\rm S}_{05}$, ${\rm S}_{05}\to {\rm R}_{04}^{0}$, ${\rm R}_{17}\to {\rm R}_{06}^{0}$, ${\rm S}_{01}^{\alpha}\to {\rm S}_{02}$ (see Tables~\ref{tab:deg12} and~\ref{tab:subdeg12}) and transitivity; note that if $A\in\mathfrak{Perm}_{(1,2)}$ and $A\to B$, then the whole closure $\overline{{\mathcal O}(A)}$ lies in $\mathfrak{Perm}_{(1,2)}$, so all the intermediate superalgebras of such a chain of degenerations also lie in $\mathfrak{Perm}_{(1,2)}$. The non-degenerations $A\not\to B$ between the superalgebras and families $A$, $B$ defining different components with ${\rm dim}\,A>{\rm dim}\,B$ are given in Tables~\ref{tab:nondeg12} and~\ref{tab:subnondeg12}; the other inclusions between the components are impossible by Section~\ref{sec:method}.
\end{proof}

\begin{corollary}\label{th:perm21} The variety $\mathfrak{Perm}_{(2,1)}$ has dimension $4$ and $7$ irreducible components:
$$
\begin{array}{@{}l@{\quad}l@{\quad}l@{\quad}l@{}}
\mathcal{C}_{01}=\overline{{\mathcal O}({\bf R}_{01})}; & \mathcal{C}_{03}=\overline{{\mathcal O}({\bf S}_{06})}; & \mathcal{C}_{05}=\overline{{\mathcal O}({\bf S}_{12})}; & \mathcal{C}_{07}=\overline{{\mathcal O}({\bf S}_{28})}.\\[4pt]
\mathcal{C}_{02}=\overline{{\mathcal O}({\bf R}_{02})}; & \mathcal{C}_{04}=\overline{{\mathcal O}({\bf S}_{11})}; & \mathcal{C}_{06}=\overline{{\mathcal O}({\bf S}_{22})};
\end{array}
$$
${\rm Rig}(\mathfrak{Perm}_{(2,1)})=\big\{{\bf R}_{01},\allowbreak\ {\bf R}_{02},\allowbreak\ {\bf S}_{06},\allowbreak\ {\bf S}_{11},\allowbreak\ {\bf S}_{12},\allowbreak\ {\bf S}_{22},\allowbreak\ {\bf S}_{28}\big\}.$
\end{corollary}
\begin{proof} Every superalgebra of $\mathfrak{Perm}_{(2,1)}$ that is not listed above belongs to one of the components by the degenerations \begin{center}
${\bf R}_{01}\to {\bf R}_{09}$, 
${\bf R}_{02}\to {\bf R}_{15}$, 
${\bf R}_{01}\to {\bf R}_{17}$, 
${\bf R}_{02}\to {\bf R}_{21}$, 
${\bf R}_{09}\to {\bf R}_{37}$,\\ 
${\bf S}_{22}\to {\bf S}_{18}$, 
${\bf S}_{22}\to {\bf S}_{20}$, 
${\bf S}_{28}\to {\bf S}_{29}$, 
${\bf S}_{29}\to {\bf S}_{24}$
\end{center} (see Tables~\ref{tab:deg21} and~\ref{tab:subdeg21}) and transitivity; note that if $A\in\mathfrak{Perm}_{(2,1)}$ and $A\to B$, then the whole closure $\overline{{\mathcal O}(A)}$ lies in $\mathfrak{Perm}_{(2,1)}$, so all the intermediate superalgebras of such a chain of degenerations also lie in $\mathfrak{Perm}_{(2,1)}$. The non-degenerations $A\not\to B$ between the superalgebras and families $A$, $B$ defining different components with ${\rm dim}\,A>{\rm dim}\,B$ are given in Tables~\ref{tab:nondeg21} and~\ref{tab:subnondeg21}; the other inclusions between the components are impossible by Section~\ref{sec:method}.
\end{proof}

\subsubsection{Associative superalgebras}
The associative superalgebras of types $(1,2)$ and $(2,1)$ are listed in Corollaries~\ref{cor:as12} and~\ref{cor:as21}.
\begin{corollary}\label{th:as12} The variety $\mathfrak{As}_{(1,2)}$ has dimension $4$ and $12$ irreducible components:
$$
\begin{array}{@{}l@{\quad}l@{\quad}l@{\quad}l@{}}
\mathcal{C}_{01}=\overline{{\mathcal O}({\rm R}_{19})}; & \mathcal{C}_{04}=\overline{{\mathcal O}({\rm S}_{03})}; & \mathcal{C}_{07}=\overline{{\mathcal O}({\rm S}_{11})}; & \mathcal{C}_{10}=\overline{{\mathcal O}({\rm S}_{14})};\\[4pt]
\mathcal{C}_{02}=\overline{{\mathcal O}({\rm R}_{24})}; & \mathcal{C}_{05}=\overline{\bigcup_{\alpha}{\mathcal O}({\rm S}_{08}^{\alpha})}; & \mathcal{C}_{08}=\overline{{\mathcal O}({\rm S}_{12})}; & \mathcal{C}_{11}=\overline{{\mathcal O}({\rm S}_{15})};\\[4pt]
\mathcal{C}_{03}=\overline{\bigcup_{\alpha}{\mathcal O}({\rm S}_{01}^{\alpha})}; & \mathcal{C}_{06}=\overline{{\mathcal O}({\rm S}_{10})}; & \mathcal{C}_{09}=\overline{{\mathcal O}({\rm S}_{13})}; & \mathcal{C}_{12}=\overline{{\mathcal O}({\rm S}_{16})}.
\end{array}
$$
${\rm Rig}(\mathfrak{As}_{(1,2)})=\big\{{\rm R}_{19},\allowbreak\ {\rm R}_{24},\allowbreak\ {\rm S}_{03},\allowbreak\ {\rm S}_{10},\allowbreak\ {\rm S}_{11},\allowbreak\ {\rm S}_{12},\allowbreak\ {\rm S}_{13},\allowbreak\ {\rm S}_{14},\allowbreak\ {\rm S}_{15},\allowbreak\ {\rm S}_{16}\big\}.$
\end{corollary}
\begin{proof} Every superalgebra of $\mathfrak{As}_{(1,2)}$ that is not listed above belongs to one of the components by the degenerations 
\begin{center}
${\rm S}_{01}^{t^{-1}}\to {\rm S}_{05}$, ${\rm S}_{05}\to {\rm R}_{04}^{0}$, ${\rm S}_{08}^{t^{-1}}\to {\rm R}_{06}^{0}$, ${\rm S}_{03}\to {\rm R}_{17}$, ${\rm S}_{01}^{\alpha}\to {\rm S}_{02}$, ${\rm S}_{03}\to {\rm S}_{04}$\end{center} (see Tables~\ref{tab:deg12} and~\ref{tab:subdeg12}) and transitivity; note that if $A\in\mathfrak{As}_{(1,2)}$ and $A\to B$, then the whole closure $\overline{{\mathcal O}(A)}$ lies in $\mathfrak{As}_{(1,2)}$, so all the intermediate superalgebras of such a chain of degenerations also lie in $\mathfrak{As}_{(1,2)}$. The non-degenerations $A\not\to B$ between the superalgebras and families $A$, $B$ defining different components with ${\rm dim}\,A>{\rm dim}\,B$ are given in Tables~\ref{tab:nondeg12} and~\ref{tab:subnondeg12}; the other inclusions between the components are impossible by Section~\ref{sec:method}.
\end{proof}

\begin{corollary}\label{th:as21} The variety $\mathfrak{As}_{(2,1)}$ has dimension $5$ and $13$ irreducible components:
$$
\begin{array}{@{}l@{\quad}l@{\quad}l@{\quad}l@{\quad}l@{}}
\mathcal{C}_{01}=\overline{{\mathcal O}({\bf R}_{01})}; & \mathcal{C}_{04}=\overline{{\mathcal O}({\bf S}_{09})}; & \mathcal{C}_{07}=\overline{{\mathcal O}({\bf S}_{12})}; & \mathcal{C}_{10}=\overline{{\mathcal O}({\bf S}_{23})}; & \mathcal{C}_{13}=\overline{{\mathcal O}({\bf S}_{31})}.\\[4pt]
\mathcal{C}_{02}=\overline{{\mathcal O}({\bf S}_{05})}; & \mathcal{C}_{05}=\overline{{\mathcal O}({\bf S}_{10})}; & \mathcal{C}_{08}=\overline{{\mathcal O}({\bf S}_{17})}; & \mathcal{C}_{11}=\overline{{\mathcal O}({\bf S}_{27})};\\[4pt]
\mathcal{C}_{03}=\overline{{\mathcal O}({\bf S}_{06})}; & \mathcal{C}_{06}=\overline{{\mathcal O}({\bf S}_{11})}; & \mathcal{C}_{09}=\overline{{\mathcal O}({\bf S}_{22})}; & \mathcal{C}_{12}=\overline{{\mathcal O}({\bf S}_{30})};
\end{array}
$$
${\rm Rig}(\mathfrak{As}_{(2,1)})=\big\{{\bf R}_{01},\allowbreak\ {\bf S}_{05},\allowbreak\ {\bf S}_{06},\allowbreak\ {\bf S}_{09},\allowbreak\ {\bf S}_{10},\allowbreak\ {\bf S}_{11},\allowbreak\ {\bf S}_{12},\allowbreak\ {\bf S}_{17},\allowbreak\ {\bf S}_{22},\allowbreak\ {\bf S}_{23},\allowbreak\ {\bf S}_{27},\allowbreak\ {\bf S}_{30},\allowbreak\ {\bf S}_{31}\big\}.$
\end{corollary}
\begin{proof} Every superalgebra of $\mathfrak{As}_{(2,1)}$ that is not listed above belongs to one of the components by the degenerations 
\begin{center}
${\bf S}_{27}\to {\bf R}_{02}$, ${\bf R}_{01}\to {\bf R}_{09}$, ${\bf R}_{02}\to {\bf R}_{15}$, ${\bf R}_{01}\to {\bf R}_{17}$, ${\bf R}_{02}\to {\bf R}_{21}$, ${\bf R}_{09}\to {\bf R}_{37}$, ${\bf S}_{22}\to {\bf S}_{18}$, ${\bf S}_{23}\to {\bf S}_{19}$, ${\bf S}_{17}\to {\bf S}_{20}$, ${\bf S}_{17}\to {\bf S}_{21}$, ${\bf S}_{27}\to {\bf S}_{25}$, ${\bf S}_{25}\to {\bf S}_{29}$, ${\bf S}_{29}\to {\bf S}_{24}$, ${\bf S}_{27}\to {\bf S}_{26}$, ${\bf S}_{27}\to {\bf S}_{28}$
\end{center} (see Tables~\ref{tab:deg21} and~\ref{tab:subdeg21}) and transitivity; note that if $A\in\mathfrak{As}_{(2,1)}$ and $A\to B$, then the whole closure $\overline{{\mathcal O}(A)}$ lies in $\mathfrak{As}_{(2,1)}$, so all the intermediate superalgebras of such a chain of degenerations also lie in $\mathfrak{As}_{(2,1)}$. The non-degenerations $A\not\to B$ between the superalgebras and families $A$, $B$ defining different components with ${\rm dim}\,A>{\rm dim}\,B$ are given in Tables~\ref{tab:nondeg21} and~\ref{tab:subnondeg21}; the other inclusions between the components are impossible by Section~\ref{sec:method}.
\end{proof}

\subsubsection{Assosymmetric superalgebras}
By the corollaries of Section~\ref{sec:cor}, the assosymmetric superalgebras of type $(1,2)$ are exactly the associative ones; hence $\mathfrak{AsSym}_{(1,2)}=\mathfrak{As}_{(1,2)}$ and its irreducible components are given in Corollary~\ref{th:as12}. The non-associative assosymmetric superalgebras of type $(2,1)$ are ${\bf S}_{07}$, ${\bf S}_{08}$, ${\bf S}_{13}$, ${\bf S}_{14}$, ${\bf S}_{15}$, ${\bf S}_{16}$, ${\bf S}_{32}$, ${\bf S}_{33}$.
\begin{corollary}\label{th:asy21} The variety $\mathfrak{AsSym}_{(2,1)}$ has dimension $5$ and $13$ irreducible components:
$$
\begin{array}{@{}l@{\quad}l@{\quad}l@{\quad}l@{\quad}l@{}}
\mathcal{C}_{01}=\overline{{\mathcal O}({\bf R}_{01})}; & \mathcal{C}_{04}=\overline{{\mathcal O}({\bf S}_{13})}; & \mathcal{C}_{07}=\overline{{\mathcal O}({\bf S}_{16})}; & \mathcal{C}_{10}=\overline{{\mathcal O}({\bf S}_{23})}; & \mathcal{C}_{13}=\overline{{\mathcal O}({\bf S}_{33})}.\\[4pt]
\mathcal{C}_{02}=\overline{{\mathcal O}({\bf S}_{07})}; & \mathcal{C}_{05}=\overline{{\mathcal O}({\bf S}_{14})}; & \mathcal{C}_{08}=\overline{{\mathcal O}({\bf S}_{17})}; & \mathcal{C}_{11}=\overline{{\mathcal O}({\bf S}_{27})};\\[4pt]
\mathcal{C}_{03}=\overline{{\mathcal O}({\bf S}_{08})}; & \mathcal{C}_{06}=\overline{{\mathcal O}({\bf S}_{15})}; & \mathcal{C}_{09}=\overline{{\mathcal O}({\bf S}_{22})}; & \mathcal{C}_{12}=\overline{{\mathcal O}({\bf S}_{32})};
\end{array}
$$
${\rm Rig}(\mathfrak{AsSym}_{(2,1)})=\big\{{\bf R}_{01},\allowbreak\ {\bf S}_{07},\allowbreak\ {\bf S}_{08},\allowbreak\ {\bf S}_{13},\allowbreak\ {\bf S}_{14},\allowbreak\ {\bf S}_{15},\allowbreak\ {\bf S}_{16},\allowbreak\ {\bf S}_{17},\allowbreak\ {\bf S}_{22},\allowbreak\ {\bf S}_{23},\allowbreak\ {\bf S}_{27},\allowbreak\ {\bf S}_{32},\allowbreak\ {\bf S}_{33}\big\}.$
\end{corollary}
\begin{proof} Every superalgebra of $\mathfrak{AsSym}_{(2,1)}$ that is not listed above belongs to one of the components by the degenerations \begin{center}
${\bf S}_{27}\to {\bf R}_{02}$, ${\bf R}_{01}\to {\bf R}_{09}$, ${\bf R}_{02}\to {\bf R}_{15}$, ${\bf R}_{01}\to {\bf R}_{17}$, ${\bf R}_{02}\to {\bf R}_{21}$, ${\bf R}_{09}\to {\bf R}_{37}$, ${\bf S}_{07}\to {\bf S}_{05}$, ${\bf S}_{08}\to {\bf S}_{06}$, ${\bf S}_{13}\to {\bf S}_{09}$, ${\bf S}_{14}\to {\bf S}_{10}$, ${\bf S}_{15}\to {\bf S}_{11}$, ${\bf S}_{16}\to {\bf S}_{12}$, ${\bf S}_{22}\to {\bf S}_{18}$, ${\bf S}_{23}\to {\bf S}_{19}$, ${\bf S}_{17}\to {\bf S}_{20}$, ${\bf S}_{17}\to {\bf S}_{21}$, ${\bf S}_{27}\to {\bf S}_{25}$, ${\bf S}_{25}\to {\bf S}_{29}$, ${\bf S}_{29}\to {\bf S}_{24}$, ${\bf S}_{27}\to {\bf S}_{26}$, ${\bf S}_{27}\to {\bf S}_{28}$, ${\bf S}_{33}\to {\bf S}_{30}$, ${\bf S}_{32}\to {\bf S}_{31}$ 
\end{center}(see Tables~\ref{tab:deg21} and~\ref{tab:subdeg21}) and transitivity; note that if $A\in\mathfrak{AsSym}_{(2,1)}$ and $A\to B$, then the whole closure $\overline{{\mathcal O}(A)}$ lies in $\mathfrak{AsSym}_{(2,1)}$, so all the intermediate superalgebras of such a chain of degenerations also lie in $\mathfrak{AsSym}_{(2,1)}$. The non-degenerations $A\not\to B$ between the superalgebras and families $A$, $B$ defining different components with ${\rm dim}\,A>{\rm dim}\,B$ are given in Tables~\ref{tab:nondeg21} and~\ref{tab:subnondeg21}; the other inclusions between the components are impossible by Section~\ref{sec:method}.
\end{proof}





\end{document}